\documentclass[preprint]{imsart}
\usepackage{amsthm,amsmath,amssymb,natbib}
\RequirePackage{hyperref}

\RequirePackage{hypernat}

\startlocaldefs

 \newcommand{\bdoi}[1]{} % Conceal this!

 \newcommand{\CutLocus}[1]{\mathcal{C}_{#1}}
 \newcommand{\Exp}{\operatorname{Exp}}
 \newcommand{\Expect}[1]{\operatorname{\mathbb{E}}\left[#1\right]}
 \newcommand{\grad}{\operatorname{grad}}
 \newcommand{\Hess}{\operatorname{Hess}}
 \newcommand{\Indicator}[1]{\operatorname{\mathbb{I}}\left[#1\right]}
 
 \newcommand{\Law}[1]{\mathcal{L}\left({#1}\right)}
 \newcommand{\Manifold}{\mathbb{M}}
 \newcommand{\Parallel}{\Pi}
 \newcommand{\Prob}[1]{\operatorname{\mathbb{P}}\left[#1\right]}

 \newcommand{\ball}{\operatorname{ball}}
 \renewcommand{\d}{\operatorname{d}}
 \DeclareMathOperator{\dist}{dist}

 \newcommand{\MetricSpace}{\mathcal{X}}
 \newcommand{\origin}{\text{\textbf{o}}}
 \newcommand{\bm}[1]{\boldsymbol{#1}}
 
 \newtheorem{lemma}{Lemma}
 \newtheorem{theorem}{Theorem}
 
 \newtheorem{corollary}{Corollary}
 \newtheorem{remark}{Remark}

\endlocaldefs

\begin{document}

\begin{frontmatter}

% \title{\texorpdfstring{Covariant Derivative, Curvature and \\Limit Theorems for Fr\'echet Means}{Covariant Derivative, Curvature and Limit Theorems for Frechet Means}}
\title{Limit theorems for empirical Fr\'echet means of independent and non-identically distributed manifold-valued random variables}
\runtitle{Limit theorems for Fr\'echet means}

\begin{aug}
\author{\fnms{Wilfrid} S. \snm{Kendall}
% \thanksref{a}
\corref{}
\ead[label=e1]{w.s.kendall@warwick.ac.uk}
\ead[label=u1,url]{http://go.warwick.ac.uk/wsk}}
\and
\author{\fnms{Huiling} \snm{Le}
% \thanksref{b}
\ead[label=e2]{Huiling.Le@nottingham.ac.uk}}

% \affiliation[a]{University of Warwick}
% \affiliation[b]{University of Nottingham}

\address[a]{Statistics Department\\
University of Warwick\\
Coventry CV4 7AL, UK\\
\printead{e1}\\
\printead{u1}}
\address[b]{School of Mathematical Sciences\\
University of Nottingham\\
University Park \\
Nottingham NG7 2RD, UK\\
\printead{e2}}

\runauthor{W. S. Kendall and H. Le}

\end{aug}

%%%%%%%%%%%%%%%%%%%%%%%%%%%%%%%%
\begin{abstract}
We prove weak laws of large numbers and central limit theorems of Lindeberg type for 
empirical centres of mass (empirical Fr\'echet means) of independent non-identically distributed random variables taking values in Riemannian manifolds. 
In order to prove
these theorems we describe and prove a simple kind of Lindeberg-Feller central approximation theorem for
vector-valued random variables, which may be of independent interest and is therefore the subject of a self-contained section.
This vector-valued result allows us to clarify the number of conditions required for the central limit theorem for empirical Fr\'echet means,
while extending its scope.
\end{abstract}


%%%%%%%%%%%%%%%%%%%%%%%%%%%%%%%%

\end{frontmatter}

\section{Introduction}\label{sec:intro}
Fr\'echet means,
or Riemannian centres of mass, were introduced at a relatively early stage of probability 
by \citet{Frechet-1948}. The idea is simple enough: generalize the mean-square characterization
of the mean \(\Expect{X}\) as the minimizer of the ``energy function'' 
\(x\mapsto \tfrac{1}{2}\Expect{(X-x)^2}\). If \(X\) takes values in a metric space \(\MetricSpace\)
this can be achieved
as follows: replace \((X-x)^2\) by the square of the distance function \(\dist(X,x)^2\).

Of course the theory of Fr\'echet means is subject to geometric complications. Uniqueness
becomes the exception rather than the rule, though existence is guaranteed if the metric
space satisfies some kind of local compactness condition. 
\citeauthor{Ziezold-1977}
(\citeyear{Ziezold-1977}, \citeyear{Ziezold-1989}, \citeyear{Ziezold-1994}) established some basic results
in this broad context, as well as developing some significant applications in applied statistics.
If the metric space \(\MetricSpace\) is specialized to a Riemannian manifold \(\Manifold\)
then it is possible to produce useful calculations and estimates using curvature; \citet{Karcher-1977}
provides a good account of this as well as surveying substantial applications
of Fr\'echet means in geometry.

Probabilistic interest in Fr\'echet means was initially spurred on by considerations of how to
generate theories of martingales taking values in manifolds, and in particular how then to
extend the mathematical application of martingale theory beyond the theory of linear elliptic
differential equations to the theory of harmonic maps \citep{Kendall-1990d,Picard-1994}.
In particular this led to strong connections with convexity theory for Riemannian manifolds,
simply expressed in \citet{Kendall-1991a} and further developed in 
\citeauthor{Kendall-1991c} (\citeyear{Kendall-1991c}, \citeyear{Kendall-1992a}, \citeyear{Kendall-1992c}) and
\citet{CorcueraKendall-1999}; more recently see \citet{Afsari-2010a}. \citet{Ziezold-1989}'s application of Fr\'echet means
to statistical shape theory has been taken up by several workers 
(see for example \citealp{Le-2001}, \citeyear{Le-2004}; also
the recent survey by \citealp{KendallLe-2010}). In particular \citeauthor{PatrangenaruBhattacharya-2003}
(\citeyear{PatrangenaruBhattacharya-2003},
\citeyear{BhattacharyaPatrangenaru-2005}) and \citet{BhattacharyaBhattacharya-2008} have developed important
statistical theory for \emph{empirical} Fr\'echet means on Riemannian manifolds, including (but not limited
to) laws of large numbers and central limit theory for independent and identically distributed 
manifold-valued random variables.

The present paper is inspired by these results of Bhattacharya and co-workers, and addresses the challenge
of extending their theory
to the non-identically distributed case. After Section \ref{sec:basic}, which establishes basic definitions
and notation, in Section \ref{sec:consistency}
we develop a weak law of large numbers for empirical Fr\'echet means in a metric space context
(Theorem \ref{thm:consistency}) which is based on the most general
possible weak law of large numbers for independent non-negative random variables (stated
here as Theorem \ref{thm:wlln-real-nonnegative-case}). In particular, we pay attention to the question of when
one can assert existence of
\emph{local} empirical Fr\'echet means lying close to a local minimizer of the aggregated energy function 
which is obtained by summing the individual
energy functions of the random variables concerned.

It is a natural step from this theory to consider central limit theorems of Lindeberg type for empirical Fr\'echet means, since
the conditions for the weak law of large numbers (Theorem \ref{thm:consistency}) involve conditions of Lindeberg type. To do this
one needs to specialize to the more specific case of Riemannian manifolds, since this allows one to use the Riemannian Exponential map
to refer the manifold to an Euclidean approximation. It is therefore apparent that a central limit theorem for the Riemannian manifold case must depend
on a central limit theorem for the random tangent vectors corresponding to the manifold-valued random variables \emph{via}
this Exponential map, and Section \ref{sec:euclid} considers the relevant theory.

In fact there is a substantial literature on central limit theorems and normal approximations for vector-valued random variables;
see \citet{BhattacharyaRao-1976} for an exposition in book form, and more recently \citet{Chatterjee-2008} and \citet{Rollin-2011} (both of whom
describe approaches which apply Stein's method). However, as we sought to generalize to a Lindeberg central limit theorem
for empirical Fr\'echet means so it became clear that we needed a subtly different result; a theorem which would describe when a sequence of
normalized random sums may be approximated by a second sequence of matching multivariate normal random variables, when there is
no guarantee of weak convergence, and \emph{when the normalization uses not individual coordinate variances but the trace of the variance-covariance
matrix of the sum}. 
These requirements mean, for example, that one cannot simply apply the Cram\'er-Wold device.
The closest general result we can find in the published literature is that of \citet[Corollary 18.2]{BhattacharyaRao-1976} 
(also see \citealp{BarbourGnedin-2009}, for specific cases arising in study of infinite occupancy schemes); however
this uses normalization in a matrix-valued sense, using the inverse of the symmetric square-root of the variance-covariance matrix (which is required to be non-singular), 
whereas we need an approach which uses scalar normalization and which can work even 
when the variance-covariance matrix degenerates.

It turns out, as we describe in Section \ref{sec:euclid}, that it is possible to formulate such a result, a multidimensional Lindeberg central \emph{approximation} theorem,
which we state and prove as Theorem \ref{thm:central-approximation-theorem} (and also Corollary \ref{thm:Feller-converse} for the Feller converse).
Proofs vary little from the classic approach of, say, \citet{Feller-1966}. However it is necessary to 
take account of the vector-valued context and to allow for a crucial intervention of the Wasserstein metric for the truncated
Euclidean distance; therefore we give the proofs in full for the sake of completeness of exposition, since the application is unfamiliar.

These results allow us to prove a Lindeberg central approximation theorem for empirical Fr\'echet means, which forms Theorem \ref{thm:lindeberg-clt} 
in Section \ref{sec:clt}. The basic idea uses Newton's root-finding algorithm, and owes much to the work of Bhattacharya and co-workers;
however 
while extending to the non-identically distributed case we are also
able to clarify the set of conditions required for the result, by exploiting the idea of central approximation rather than central limits, and we can derive
a rather explicit form for the variance-covariance matrices of the approximating multivariate normal random variables. The paper concludes with a small
number of illustrative examples, demonstrating how the results simplify in the case of independent and identically distributed random variables, and also in the
case when the Riemannian manifold is of constant sectional curvature, or carries a K\"ahler structure with constant holomorphic sectional curvature.

\subsection*{Acknowledgements}
We gratefully acknowledge the helpful advice of A.~Barbour, N.H.~Bingham, C.M.~Goldie, P.~Hall and R.~Bhattacharya.

\section{Basic theory and notation}\label{sec:basic}
Consider the \emph{energy function} of a random variable \(X\) taking values in a metric space \(\mathcal{X}\):
\[
 \phi(x) \quad=\quad\Expect{\frac{1}{2}\dist(X,x)^2}\,.
\]
Observe that if \(\phi\) is finite at one point of \(\mathcal{X}\) then it is finite everywhere, by an argument using the triangle inequality.
Given independent \(X_1\), \ldots, \(X_n\), the \emph{aggregate energy function} is simply the sum
\[
 \phi_n(x) \quad=\quad\sum_{m=1}^n\Expect{\frac{1}{2}\dist(X_m,x)^2}\,.
\]

A \emph{Fr\'echet mean} is a global minimizer of \(\phi\). Note that there can be more than one Fr\'echet mean: we then consider \emph{the set of Fr\'echet means}
\[
 \arg\min_x \Expect{\frac{1}{2}\dist(X,x)^2}\,.
\]
An \emph{empirical Fr\'echet mean} is a global minimizer of the energy function based on the empirical probability measure defined by a sample \(X_1\), \ldots, \(X_n\):
thus \emph{the set of empirical Fr\'echet means} is
\[
 \arg\min_x \frac{1}{n}\sum_{i=1}^n\frac{1}{2}\dist(X_i,x)^2\,.
\]
(In case of local compactness, the existence of global minimizers of both kinds follows immediately from
\(\dist(X,y)+\dist(X,x)\geq\dist(x,y)\).)

Some of our results hold for local minimizers; we use the term \emph{local Fr\'echet mean} to describe a local minimizer of \(\phi\), 
while a \emph{local empirical Fr\'echet mean} denotes a local minimizer of the energy function based on the empirical probability measure defined by a sample \(X_1\), \ldots, \(X_n\)
of points from the metric space \(\MetricSpace\).

We shall use the operator-theoretic notation \(\Expect{H}\) to denote the expectation of a random variable \(H\). In particular
we shall write \(\Expect{H\;;\;A}=\Expect{H\Indicator{A}}\), where \(\Indicator{A}\) is the indicator random variable for an event \(A\).

\section{Weak law of large numbers for empirical Fr\'echet means}\label{sec:consistency}
\citet{Ziezold-1977} established a strong law of large numbers for sequences of independent identically distributed random variables \(X_1\), \(X_2\), \ldots
taking values in a separable metric space \(\MetricSpace\) (actually \citeauthor{Ziezold-1977} covered the more general case of a separable \emph{finite quasi-metric space}).
Imposing the condition that the energy function \(\Expect{\tfrac{1}{2}\dist(X_i,x)^2}\) be finite for some (and thus all) \(x\),
\citeauthor{Ziezold-1977} was then able to show that almost surely the limit of the closure of the sup of the set of empirical Fr\'echet means is a subset of the set of Fr\'echet means
(up to an event of zero probability measure):
\begin{equation}\label{eq:Ziezold}
 \bigcap_{k=1}^\infty \overline{\bigcup_{n=k}^\infty \arg\min_x \frac{1}{n}\sum_{i=1}^n\frac{1}{2}\dist(X_i,x)^2}
\quad\subseteq\quad
\arg\min_x \Expect{\frac{1}{2}\dist(X_1,x)^2}\,.
\end{equation}
Here of course the \(\arg\min\) are treated as random closed sets.

If \(\MetricSpace\) is not compact then it is possible for a sequence of empirical Fr\'echet means to diverge to infinity even when \eqref{eq:Ziezold} holds.
Given uniqueness of the Fr\'echet mean, \citet[Theorem 2.3]{PatrangenaruBhattacharya-2003} have shown that a strong law of large numbers follows from imposition of the additional
condition that every closed bounded subset of \(\mathcal{X}\) is compact; in that case every sequence of measurable choices from 
the sets
\[
\arg\min_x \frac{1}{n}\sum_{i=1}^n\frac{1}{2}\dist(X_i,x)^2 
\]
of empirical Fr\'echet means will almost surely converge to the unique Fr\'echet mean.

In this section we derive a \emph{weak} law of large numbers in the more general case of non-identically distributed independent random variables \(X_1\), \(X_2\), \ldots,
taking values in a separable metric space \(\MetricSpace\) possessing the bounded compactness property of \citeauthor{PatrangenaruBhattacharya-2003}, and
such that the individual energy functions \(\Expect{\tfrac{1}{2}\dist(X_n,x)^2}\) are finite for some (and therefore for all) \(x\in\MetricSpace\). 
Evidently we need to impose extra conditions to compensate for the lack of identical distribution; 
we will require that the aggregate energy function \(\phi_n(x)=\sum_{i=1}^n \Expect{\tfrac{1}{2}\dist( X_i,x)^2}\) has a strict local minimum near a fixed 
reference point \(\origin\in\MetricSpace\),
and we will require that this holds uniformly as \(n\to\infty\) (in a particular sense captured in 
the displayed equation
\eqref{eq:strict-local-minimizer} 
below). 
In recompense for this restriction, our results describe the behaviour of \emph{local} empirical Fr\'echet means lying in a geodesic ball
\(\ball(\origin,\rho_1)\subseteq \MetricSpace\). 
The particular uniformity requirement is that for each positive \(\rho_0\leq\rho_1\) there is positive \(\kappa=\kappa(\rho_0,\rho_1)\) such that, for all \(n\),
\begin{equation}\label{eq:strict-local-minimizer}
(1+\kappa)\phi_n(\origin)\quad<\quad
\inf\left\{\phi_n(y): \rho_0\leq\dist(y,\origin)\leq\rho_1\right\}\,.
\end{equation}

Bearing in mind that the ultimate aim of this paper
is to prove a central limit theorem, convergence in probability is a more natural objective than almost sure convergence. 
Therefore it is reasonable to restrict attention to the weaker notion of convergence in probability. 
Moreover even in the scalar case the law-of-large-numbers conditions for convergence in probability are clearer and more easily stated than for convergence almost surely.
The key theorem for our treatment is the weak law of large numbers for non-identically distributed non-negative real random variables.
% ; see for example
% \citet[Chapter 10, Theorem 1, Corollary 2]{ChowTeicher-2003}. 
We state a special case of this result:
\begin{theorem}\label{thm:wlln-real-nonnegative-case}
Suppose that \(Z_1\), \(Z_2\), \ldots are independent % \marginpar{can `non-identically' be dropped?} 
non-negative real random variables, not necessarily of the same distribution. 
Suppose further that
\begin{equation}\label{eq:real-Lindeberg-type}
 \frac{1}{\sum_{r=1}^n\Expect{Z_r}}\sum_{m=1}^n \Expect{Z_m \;;\; Z_m \geq \varepsilon \sum_{r=1}^n\Expect{Z_r}}\quad\longrightarrow\quad 0
\qquad\text{for each } \varepsilon>0\,.
\end{equation}
Then it is the case that as \(n\to\infty\) so
\begin{equation}\label{eq:wlln-real}
 \frac{\sum_{r=1}^n Z_r}{\sum_{r=1}^n\Expect{Z_r}}\quad\longrightarrow\quad 1\qquad\text{in probability}\,.
\end{equation}
\end{theorem}
\noindent
This theorem follows directly from \citet[Chapter 10, Theorem 1, Corollary 2]{ChowTeicher-2003}.

\begin{remark}
The condition \eqref{eq:real-Lindeberg-type} can be viewed as an equation of Lindeberg type. 
Indeed, if \(W_1\), \(W_2\), \ldots are independent real random variables with \(\Expect{W_m}=0\) and such that \(W_m^2=Z_m\), then
\eqref{eq:real-Lindeberg-type} corresponds exactly to the usual Lindeberg condition for the sequence \(\{W_m:m\geq1\}\).
Thus \citet[Chapter 10, Theorem 1, Corollary 2]{ChowTeicher-2003} signals the close connection between weak laws of large numbers and the central limit theorem.
\end{remark}

Our strategy for proving a weak law of large numbers for non-identically distributed \(\MetricSpace\)-valued random variables is as follows:
consider the condition \eqref{eq:real-Lindeberg-type} applied to the case \(Z_m^{(x)}=\tfrac{1}{2}\dist(X_m,x)^2\),
and then apply
the corresponding weak laws of large numbers \eqref{eq:wlln-real}.
Under suitable additional conditions the aggregate empirical energy functions 
\(\sum_{m=1}^n Z_m^{(x)}=\sum_{m=1}^n\tfrac{1}{2}\dist(X_m,x)^2\)
can be made to approximate the 
aggregate energy functions \(\phi_n(x)\)
closely enough to ensure that the uniform local minimum property
% sufficiently closely that 
% \eqref{eq:strict-local-minimizer} 
forces convergence to \(1\) of the probability of there being local empirical Fr\'echet means close to \(\origin\).

For a useful result it is preferable to require that the Lindeberg-type condition apply only at the chosen reference point \(\origin\). 
For a general metric space \(\MetricSpace\) we should not expect the Lindeberg-type condition for the \(Z_m^{(\origin)}\) to imply the corresponding conditions
obtained when \(\origin\) is replaced by a general \(x\in\MetricSpace\). However we can prove a partial result in this direction, which will be sufficient for our purposes:
\begin{lemma}\label{lem:partial-globalization}
Suppose as above that \(\MetricSpace\) is a separable metric space.
Let \(X_1\), \(X_2\), \ldots be independent \(\MetricSpace\)-valued random variables with finite energy functions.
 The following conditions of Lindeberg-type are equivalent:
\begin{align}
&\text{Firstly, a local Lindeberg condition: }\quad\nonumber\\
&\frac{1}{\phi_n(x)}\sum_{m=1}^n\Expect{\frac{1}{2}\dist(X_m,x)^2\;;\; \frac{1}{2}\dist(X_m,x)^2>\varepsilon\phi_n(x)} \quad\to\quad0
\nonumber\\
&\qquad\text{as }n\to\infty\text{ for each }\varepsilon>0\,.
\label{eq:lindeberg-local}
\\
&\text{Secondly, a semi-global Lindeberg condition: }\quad\nonumber\\
&\frac{1}{n \phi_n(x)}
\sum_{i=1}^n\sum_{j=1}^n
\Expect{\frac{1}{2}\dist(X_i,X_j)^2\;;\; \frac{1}{2}\dist(X_i,X_j)^2>\varepsilon\phi_n(x)} \quad\to\quad0
\nonumber\\
&\qquad\text{as }n\to\infty\text{ for each }\varepsilon>0\,.
\label{eq:lindeberg-global}
\end{align}
\end{lemma}

\begin{remark}
 Note that the presence of \(\phi_n(x)\) in \eqref{eq:lindeberg-global} means that this semi-global condition is \emph{not} truly global,
since \(\phi_n(x)=\sum_{m=1}^n\Expect{\tfrac{1}{2}\dist(X_m,x)^2}\) depends implicitly on the choice of \(x\in\MetricSpace\).
\end{remark}
\begin{proof}
First suppose that the local condition \eqref{eq:lindeberg-local} holds. 
We shall use this to produce an upper bound on the quantity on the left-hand side of \eqref{eq:lindeberg-global}.
Indeed
\begin{multline*}
 \frac{1}{n \phi_n(x)}
\sum_{i=1}^n\sum_{j=1}^n
\Expect{\frac{1}{2}\dist(X_i,X_j)^2\;;\; \frac{1}{2}\dist(X_i,X_j)^2>\varepsilon\phi_n(x)}
\quad\leq\quad\\
 \frac{4}{\phi_n(x)}
\sum_{i=1}^n
\Expect{\frac{1}{2}\dist(X_i,x)^2\;;\; \frac{1}{2}\dist(X_i,x)^2>\frac{\varepsilon}{4}\phi_n(x)} +\\
+
\frac{4}{n}
\sum_{i=1}^n
\Prob{\frac{1}{2}\dist(X_i,x)^2>\frac{\varepsilon}{4}\phi_n(x)}\,.
\end{multline*}
Here we make direct use of the triangle inequality \emph{via}
\[
\dist(X_i,X_j)^2
\quad\leq\quad 2\dist(X_i,x)^2 + 2\dist(X_j,x)^2 \,;
\]
in particular the condition that
\(\dist(X_i,X_j)>\sqrt{2\varepsilon\phi_n(x)}\) implies that at least one of
\(\dist(X_i,x)>\tfrac{1}{2}\sqrt{2\varepsilon\phi_n(x)}\) or \(\dist(X_j,x)>\tfrac{1}{2}\sqrt{2\varepsilon\phi_n(x)}\)
must hold.

The Markov inequality implies that
\begin{multline*}
 \frac{4}{n}
\sum_{i=1}^n
\Prob{\frac{1}{2}\dist(X_i,x)^2>\frac{\varepsilon}{4}\phi_n(x)}
\quad\leq\quad\\
\frac{16}{\varepsilon n\phi_n(x)}
\sum_{i=1}^n
\Expect{\frac{1}{2}\dist(X_i,x)^2\;;\;\frac{1}{2}\dist(X_i,x)^2>\frac{\varepsilon}{4}\phi_n(x)}
\end{multline*}
and therefore we obtain
\begin{multline*}
 \frac{1}{n \phi_n(x)}
\sum_{i=1}^n\sum_{j=1}^n
\Expect{\frac{1}{2}\dist(X_i,X_j)^2\;;\; \frac{1}{2}\dist(X_i,X_j)^2>\varepsilon\phi_n(x)}
\quad\leq\quad\\
 4\left(1+\frac{4}{\varepsilon n}\right)\frac{1}{\phi_n(x)}
\sum_{i=1}^n
\Expect{\frac{1}{2}\dist(X_i,x)^2\;;\; \frac{1}{2}\dist(X_i,x)^2>\frac{\varepsilon}{4}\phi_n(x)}\,.
\end{multline*}
For any \(\varepsilon>0\) this upper bound tends to zero as \(n\to\infty\), by \eqref{eq:lindeberg-local}, and therefore
we obtain \eqref{eq:lindeberg-global}.

\bigskip

Now suppose on the other hand that the semi-global condition \eqref{eq:lindeberg-global} holds. 
If \(\dist(X_i,x)>\sqrt{2\varepsilon\phi_n(x)}\) and \(\dist(X_j,x)\leq\tfrac{1}{2}\sqrt{2\varepsilon\phi_n(x)}\)
then it follows that \(\dist(X_i,X_j)>\tfrac{1}{2}\sqrt{2\varepsilon\phi_n(x)}\).
We deduce that
\begin{multline*}
%   \frac{1}{n \phi_n(x)}
% \sum_{i=1}^n\sum_{j=1}^n
\Expect{\frac{1}{2}\dist(X_i,X_j)^2\;;\; \frac{1}{2}\dist(X_i,X_j)^2>\frac{\varepsilon}{4}\phi_n(x)}
\quad\geq\quad\\
%   \frac{1}{n \phi_n(x)}
% \sum_{i=1}^n\sum_{j=1}^n
\Expect{\frac{1}{2}\dist(X_i,X_j)^2\;;\; \dist(X_i,x)>\sqrt{2\varepsilon\phi_n(x)}, \dist(X_j,x)\leq\frac{1}{2}\sqrt{2\varepsilon\phi_n(x)}
}\,.
\end{multline*}
If \(\dist(X_i,x)>\sqrt{2\varepsilon\phi_n(x)}\) and \(\dist(X_j,x)\leq\tfrac{1}{2}\sqrt{2\varepsilon\phi_n(x)}\)
then \(\dist(X_i,X_j)\geq \dist(X_i,x)-\dist(X_j,x)\geq \tfrac{1}{2}\sqrt{2\varepsilon\phi_n(x)}\geq\tfrac{1}{2}\dist(X_i,x)\),
and so
\begin{multline*}
   \frac{1}{n \phi_n(x)}
\sum_{i=1}^n\sum_{j=1}^n
\Expect{\frac{1}{2}\dist(X_i,X_j)^2\;;\; \frac{1}{2}\dist(X_i,X_j)^2>\frac{\varepsilon}{4}\phi_n(x)}
\quad\geq\quad\\
  \frac{1}{4 \phi_n(x)}
\sum_{i=1}^n
\Bigg(
\Expect{\frac{1}{2}\dist(X_i,x)^2\;;\; \frac{1}{2}\dist(X_i,x)^2>\varepsilon\phi_n(x)}
\times\\
\times
  \frac{1}{n}
\sum_{j=1,j\neq i}^n
\Prob{\frac{1}{2}\dist(X_j,x)^2\leq\frac{\varepsilon}{4}\phi_n(x)}
\Bigg)\,.
\end{multline*}
Finally we take complements and use Markov's inequality to deduce
\begin{multline*}
 \frac{1}{n}
\sum_{j=1,j\neq i}^n
\Prob{\frac{1}{2}\dist(X_j,x)^2\leq\frac{\varepsilon}{4}\phi_n(x)}
% =1 - \frac{1}{n}-
%  \frac{1}{n}
% \sum_{j=1,j\neq i}^n
% \Prob{\frac{1}{2}\dist(X_j,x)^2>\frac{\varepsilon}{4}\phi_n(x)}
\\
\quad \geq\quad
1 - \frac{1}{n}-
 \frac{4}{n\varepsilon\phi_n(x)}
\sum_{j=1,j\neq i}^n
\Expect{\frac{1}{2}\dist(X_j,x)^2\;;\;\frac{1}{2}\dist(X_j,x)^2>\frac{\varepsilon}{4}\phi_n(x)}
\\
\quad\geq\quad
1 - \frac{1}{n}-
 \frac{4}{n\varepsilon}
\quad\geq\quad
\frac{1}{2}
\qquad\text{once }
n\geq 2(1+\tfrac{4}{\varepsilon})\,.
\end{multline*}
Taking \(n\geq 2(1+\tfrac{4}{\varepsilon})\), we deduce that \eqref{eq:lindeberg-global} implies \eqref{eq:lindeberg-local} by arguing that
\begin{multline*}
    \frac{1}{n \phi_n(x)}
\sum_{i=1}^n\sum_{j=1}^n
\Expect{\frac{1}{2}\dist(X_i,X_j)^2\;;\; \frac{1}{2}\dist(X_i,X_j)^2>\frac{\varepsilon}{4}\phi_n(x)}
\quad\geq\quad\\
 \frac{1}{8} \frac{1}{\phi_n(x)}
\sum_{i=1}^n
\Expect{\frac{1}{2}\dist(X_i,x)^2\;;\; \frac{1}{2}\dist(X_i,x)^2>\varepsilon\phi_n(x)}\,.
\end{multline*}

This establishes the equivalence of local and semi-global conditions.
\end{proof}

Effective use of the semi-global Lindeberg condition depends on a 
lower bound on the growth
of the energy function \(\phi_n(y)\) as \(\dist(\origin,y)\) increases.
\begin{lemma}\label{lem:growth-condition}
Suppose as above that \(\MetricSpace\) is a separable metric space.
Let \(X_1\), \(X_2\), \ldots be \(\MetricSpace\)-valued random variables with finite energy functions. 
Suppose that the aggregate energy function 
\(
\phi_n(y)=\sum_{m=1}^n\Expect{\tfrac{1}{2}\dist(X_m,y)^2}
\)
attains its minimum over \(\MetricSpace\) at \(y=\origin\):
\begin{equation} \label{eq:local-min0}
 \phi_n(\origin)\quad\leq\quad 
\phi_n(y)\,.
\end{equation}
Then the aggregate energy function grows at least linearly at any \(y\neq\origin\):
\begin{equation}\label{eq:growth-condition}
 \phi_n(y) \quad\geq\quad \frac{\dist(y,\origin)^2}{16} n\,.
\end{equation}
\end{lemma}

\begin{proof}
For convenience, set \(\rho=\dist(y,\origin)\).
If \(\phi_n(\origin)\geq \rho^2 n/16\) then \eqref{eq:growth-condition} follows from inequality \eqref{eq:local-min0}. So we can suppose that \(\phi_n(\origin)<\rho^2 n/16\).

For additional convenience let \(M\) be a random integer chosen uniformly from \(\{1, 2, \ldots, n\}\)
(independently of \(X_1\), \ldots, \(X_n\)).
Then
\begin{align*}
 \frac{1}{n}\phi_n(y) \quad=\quad \Expect{\frac{1}{2}\dist(X_M,y)^2}
\quad&\geq\quad \frac{\rho^2}{8}\Prob{\frac{1}{2}\dist(X_M,y)^2\geq \frac{\rho^2}{8}} \\
& \qquad\qquad \text{(Markov inequality)} \\
\;=\; \frac{\rho^2}{8}\Prob{\dist(X_M,y)\geq \frac{\rho}{2}}
\;&\geq\; \frac{\rho^2}{8}\Prob{\dist(X_M,\origin)< \frac{\rho}{2}}\\
& \qquad\qquad \text{(triangle inequality)} \\
\,=\,\frac{\rho^2}{8}\left(1-\Prob{\frac{1}{2}\dist(X_M,\origin)^2\geq \frac{\rho^2}{8}}\right)
\;\geq&\; \frac{\rho^2}{8}\left(1-\frac{8}{\rho^2}\Expect{\frac{1}{2}\dist(X_M,\origin)^2}\right)\\
& \qquad \text{(Markov inequality again)} \\
\quad\geq\quad \frac{\rho^2}{8}\left(1-\frac{8}{\rho^2}\frac{\rho^2}{16}\right)
\quad&\geq\quad \frac{\rho^2}{16}\,.
\end{align*}
So \eqref{eq:growth-condition} follows in this case also.
\end{proof}

We are now in a position to state and prove the main result of this section. We follow \citet{PatrangenaruBhattacharya-2003}
by imposing the compactness of bounded closed sets, and also impose the uniform local minimum property
described above by Inequality \eqref{eq:strict-local-minimizer}.

\begin{theorem}\label{thm:consistency} 
Suppose \(\MetricSpace\) is a separable metric space for which all bounded closed sets are compact.
Let \(X_1\), \(X_2\), \ldots be independent non-identically distributed \(\MetricSpace\)-valued random variables
such that \(\Expect{\tfrac{1}{2}\dist(X_m,\origin)^2}<\infty\)
for a given reference point \(\origin\in\MetricSpace\) (hence for all points in \(\MetricSpace\)),  for each \(m\).
Suppose also that 
the uniform local minimum property obtains:
there is fixed finite \(\rho_1>0\) such that Inequality \eqref{eq:strict-local-minimizer} holds for each positive \(\rho_0\leq\rho_1\).
Thus
% the 
there is \(\kappa=\kappa(\rho_0,\rho_1)\) such that \((1+\kappa)\phi_n(\origin)\) (for the
aggregate energy function \(\phi_n\) specified above) is 
a strict lower bound for the values of \(\phi_n\) on
% relatively strictly smaller at \(\origin\)
% than 
the annulus centred at \(\origin\) and defined by radii \(\rho_0\), \(\rho_1\).
% ; to be specific, there is \(\kappa=\kappa(\rho_0,\rho_1)>0\) such that
% for each \(n\)
% \begin{equation} \label{eq:local-min}
%  (1+\kappa)\phi_n(\origin)\quad<\quad \inf \{\phi_n(y): \rho_0\leq\dist(y,\origin)\leq\rho_1\}\,.
% \end{equation}
Finally, suppose that the \(X_m\) satisfy a local condition of Lindeberg type at \(\origin\): for each \(\varepsilon>0\), as \(n\to\infty\) so
% eventually \(\phi_n(\origin)>0\) and moreover
\begin{equation}\label{eq:lindeberg}
\frac{1}{\phi_n(\origin)}\sum_{m=1}^n\Expect{\dist(X_m,\origin)^2\;;\;\dist(X_m,\origin)^2>\varepsilon \phi_n(\origin)}\longrightarrow0\,.
\end{equation}
Consider any measurable choice of a sequence of local minimizers
\[
 \mathcal{E}(X_1,\ldots,X_n)\quad=\quad\mathop{\arg\inf}_{x\in\ball(\origin,\rho_1)}\left\{\sum_{m=1}^n\frac{1}{2}\dist(X_m,x)^2\right\}\,.
\]
There exists at least one such sequence such that
\[
 \Prob{\mathcal{E}(X_1,\ldots,X_n)\in\ball(\origin,\rho_0)}\quad\to\quad1\,,
\]
and for any such sequence 
\(\mathcal{E}(X_1,\ldots,X_n)\to\origin\) in probability.
\end{theorem}

\begin{proof}
First note that global (and hence also local) minimizers of the aggregate empirical energy function always exist and are confined to an almost surely bounded region: indeed global minimizers 
for the sample \(X_1\), \ldots, \(X_n\) are simply conventional 
Fr\'echet means of the \(n\)-point empirical distribution, and the argument of \citet[Theorem 2.1]{PatrangenaruBhattacharya-2003} applies (this theorem
is stated for Riemannian manifolds, but the portion relating to existence within a bounded region is a purely metric space argument, using the compactness of bounded sets).

Evidently it suffices to show that Inequality \eqref{eq:strict-local-minimizer} has high probability of being replicated at the empirical level: 
it is enough to show that the following probability converges to \(1\) as \(n\to\infty\) 
% we need to show that, 
for each positive \(\rho_0<\rho_1\):
% , the following probability converges to \(1\) as \(n\to\infty\):
\begin{equation}\label{eq:target}
\Prob{
\sum_{m=1}^n\frac{1}{2}\dist(X_m,\origin)^2 < \inf\left\{\sum_{m=1}^n\frac{1}{2}\dist(X_m,y)^2\;;\; \rho_0\leq \dist(y,\origin)\leq\rho_1\right\}
}\,. 
\end{equation}
For then it follows immediately that any sequence of local minimizers of the aggregate empirical energy function \emph{restricted to \(\ball(\origin,\rho_1)\)}
\[
 \mathcal{E}(X_1,\ldots,X_n)\quad=\quad\mathop{\arg\inf}_{x\in\ball(\origin,\rho_1)}\left\{\sum_{m=1}^n\frac{1}{2}\dist(X_m,x)^2\right\}\,,
\]
must (as \(n\to\infty\)) eventually have
arbitrarily high probability of lying in \(\ball(\origin,\rho_0)\), and must in this event
be a local minimizer of the unrestricted aggregate empirical energy function. Since \eqref{eq:target} holds for each positive \(\rho_0<\rho_1\), we may deduce
that \(\dist(\mathcal{E}(X_1,\ldots,X_n),\origin)\to0\) in probability.

To begin the proof, first note that the result follows trivially if \(\phi_n(\origin)=0\) for all \(n\), for then \(X_m=\origin\) almost surely for all \(n\).
Otherwise by
Theorem \ref{thm:wlln-real-nonnegative-case}
\begin{equation}\label{eq:convergence in centre}
  \frac{1}{\phi_n(\origin)}\sum_{m=1}^n\frac{1}{2}\dist(X_m,\origin)^2\quad\to\quad 1
\quad\text{ in probability.}
\end{equation}
% (Here and in the following we work with convergence in probability.)
Furthermore Lemma \ref{lem:partial-globalization} and \eqref{eq:lindeberg} show that,
for each \(\varepsilon>0\), as \(n\to\infty\) so
\[
 \frac{1}{n \phi_n(\origin)}
\sum_{i=1}^n\sum_{j=1}^n
\Expect{\frac{1}{2}\dist(X_i,X_j)^2\;;\; \frac{1}{2}\dist(X_i,X_j)^2>\varepsilon\phi_n(\origin)} \quad\to\quad0\,.
\]
Moreover \eqref{eq:strict-local-minimizer} implies that if \(\rho_0\leq\dist(y,\origin)\leq\rho_1\) then also
\[
 \frac{1}{n \phi_n(y)}
\sum_{i=1}^n\sum_{j=1}^n
\Expect{\frac{1}{2}\dist(X_i,X_j)^2\;;\; \frac{1}{2}\dist(X_i,X_j)^2>\varepsilon\phi_n(y)} \quad\to\quad0\,.
\]
A further application of Lemma \ref{lem:partial-globalization} then
shows that, for each \(\varepsilon>0\), as \(n\to\infty\)
\[
 \frac{1}{\phi_n(y)}\sum_{m=1}^n\Expect{\frac{1}{2}\dist(X_m,y)^2\;;\; \frac{1}{2}\dist(X_m,y)^2>\varepsilon\phi_n(y)}\quad\to\quad0\,.
\]
Consequently we may also deduce that if \(\rho_0\leq\dist(y,\origin)\leq\rho_1\) then
\[
 \frac{1}{\phi_n(y)}\sum_{m=1}^n\frac{1}{2}\dist(X_m,y)^2\quad\to\quad 1
\qquad\text{in probability.}
\]

Now we have established suitable convergence in probability for the energy functions, but only holding pointwise not uniformly. 
Were we able to uniformize this over the whole of 
the annulus \(A(\rho_0,\rho_1)=\{y:\rho_0\leq\dist(y,\origin)\leq\rho_1\}\),
and were we able to overcome the distinction between \(\phi_n(\origin)\) and \(\phi_n(y)\) for \(y\in A(\rho_0,\rho_1)\), then
we would achieve the required convergence for \eqref{eq:target} \emph{via} Inequality \eqref{eq:strict-local-minimizer}. 
Following \cite{PatrangenaruBhattacharya-2003},
we do this by selecting \(y_1\), \ldots, \(y_k\) from \(A(\rho_0,\rho_1)\) to form
a finite \(\delta\)-net for \(A(\rho_0,\rho_1)\), for suitably small \(\delta>0\). Consider two points \(y\), \(z\in A(\rho_0,\rho_1)\)
with \(\dist(y,z)<\delta\). Then we can use \(\dist(X_m,y)\leq 1+\dist(X_m,y)^2\) to deduce
\[
 \dist(X_m,z)^2 \quad\leq\quad (\dist(X_m,y)+\delta)^2
\quad\leq\quad
(1+2\delta)\dist(X_m,y)^2 + (2+\delta)\delta\,,
\]
likewise
\[
  \dist(X_m,y)^2 
\quad\leq\quad
(1+2\delta)\dist(X_m,z)^2 + (2+\delta)\delta\,.
\]
Applying this to whichever is the larger of \(\dist(X_m,y)^2 \), \(\dist(X_m,z)^2\), and then using \(\dist(X_m,z)^2\leq(\dist(X_m,y)+\delta)^2\leq2\dist(X_m,y)^2+2\delta^2\),
\begin{multline*}
 |\dist(X_m,z)^2-\dist(X_m,y)^2|
\quad\leq\quad\\
2\delta\max\{\dist(X_m,z)^2,\dist(X_m,y)^2\} + (2+\delta)\delta
\quad\leq\quad\\
4\delta\dist(X_m,y)^2 + 4\delta^3  + (2+\delta)\delta
\quad=\quad
\left(4\dist(X_m,y)^2 + 4\delta^2+\delta  + 2\right)\delta\,.
\end{multline*}
For \(z\in A(\rho_0,\rho_1)\), choose \(p(z)\) to be an element of the \(\delta\)-net which is closest to \(z\). Then the above implies that
\begin{multline*}
 \sup\left\{\left|1-\frac{1}{\phi_n(p(z))}\sum_{m=1}^n\frac{1}{2}\dist(X_m,z)^2\right|\;;\; z\in A(\rho_0,\rho_1)\right\}
\quad\leq\quad\\
\max_{i=1,\ldots,k}\Bigg\{\left|1-\frac{1}{\phi_n(y_i)}\sum_{m=1}^n\frac{1}{2}\dist(X_m,y_i)^2\right| + \qquad\\
+ \left(\frac{4}{\phi_n(y_i)}\sum_{m=1}^n\frac{1}{2}\dist(X_m,y_i)^2\ + \frac{4\delta^2+\delta  + 2}{2}\frac{n}{\phi_n(y_i)}\right)\delta
\Bigg\}\,.
\end{multline*}

Thus we establish useful limiting bounds holding in probability as \(n\to\infty\) so long as we can show that if \(y\in A(\rho_0,\rho_1)\) then
\[
 \liminf_n \frac{\phi_n(y)}{n} \quad>\quad 0.
\]
But this follows (with an explicit lower bound) from Lemma \ref{lem:growth-condition}: hence
\begin{multline*}
 \sup\left\{\left|1-\frac{1}{\phi_n(p(z))}\sum_{m=1}^n\frac{1}{2}\dist(X_m,z)^2\right|\;;\; z\in A(\rho_0,\rho_1)\right\}
\quad\leq\quad\\
\max_{i=1,\ldots,k}\Bigg\{\left|1-\frac{1}{\phi_n(y_i)}\sum_{m=1}^n\frac{1}{2}\dist(X_m,y_i)^2\right| +\\
+ \left(\frac{4}{\phi_n(y_i)}\sum_{m=1}^n\frac{1}{2}\dist(X_m,y_i)^2\ + \frac{4\delta^2+\delta  + 2}{2}\frac{16}{\rho_0^2}\right)\delta
\Bigg\}\,.
\end{multline*}
Consequently, once \(\rho_0\) is fixed, for any \(\varepsilon>0\)  we can choose \(\delta\) small enough so that with probability
tending to \(1\) as \(n\to\infty\)
\[
  \sup\left\{\left|1-\frac{1}{\phi_n(p(z))}\sum_{m=1}^n\frac{1}{2}\dist(X_m,z)^2\right|\;;\; z\in A(\rho_0,\rho_1)\right\}
\quad\leq\quad\frac{\varepsilon}{2}\,.
\]
We now use \eqref{eq:strict-local-minimizer} to deduce that with probability
tending to \(1\) as \(n\to\infty\)
\begin{multline*}
 \inf\left\{\sum_{m=1}^n\frac{1}{2}\dist(X_m,z)^2\;;\; z\in A(\rho_0,\rho_1)\right\}\quad\geq\quad\\
(1+\kappa)\phi_n(\origin) \inf\left\{\frac{1}{\phi_n(p(z))}\sum_{m=1}^n\frac{1}{2}\dist(X_m,z)^2\;;\; z\in A(\rho_0,\rho_1)\right\}\quad\geq\quad\\
% (1-\frac{\varepsilon}{2})\max_{i=1,\ldots,k}\left\{\phi_n(y_i)\right\}\\
\quad\geq\quad(1-\frac{\varepsilon}{2})(1+\kappa)\phi_n(\origin)
\quad\geq\quad(1-{\varepsilon})(1+\kappa)\sum_{m=1}^n\frac{1}{2}\dist(X_m,\origin)^2
\end{multline*}
where the last step uses the convergence in probability noted in \eqref{eq:convergence in centre}. 
This establishes that the quantity in \eqref{eq:target} must converge to \(1\); this completes the proof of the theorem.
\end{proof}

We have therefore shown that sequences of local empirical Fr\'echet means must converge in probability to a reference point \(\origin\)
when this reference point is uniformly a strict local minimum of the aggregate energy function so long as a condition of Lindeberg-type
is satisfied at \(\origin\). Under the additional condition of a linear bound on the growth of \(\phi_n(\origin)\) it is possible also to control the behaviour of
global minimizers and derive a result for global empirical Fr\'echet means.

\begin{corollary}
 In the situation of Theorem \ref{thm:consistency}, suppose that condition \eqref{eq:strict-local-minimizer} holds for all positive \(\rho_1\)
(thus in particular \(\origin\) is the unique global Fr\'echet mean), and suppose
in addition that there is a positive constant \(C\) such that
\begin{equation}\label{eq:energy-growth}
 \limsup_{n\to\infty} \frac{1}{n}\phi_n(\origin) = \limsup_{n\to\infty} \frac{1}{n}\sum_{m=1}^n\Expect{\frac{1}{2}\dist(X_m,\origin)^2}
\quad\leq\quad C^2\,.
\end{equation}
Then any measurably selected sequence of local empirical Fr\'echet means converges to \(\origin\) in probability.
\end{corollary}

\begin{proof}
Following the proof of Theorem \ref{thm:consistency}, it would suffice to show that, for sufficiently large \(\rho_1\),
\[
 \Prob{\frac{1}{n}\sum_{m=1}^n \frac{1}{2}\dist(X_m,\origin)^2 + 1\leq
\inf\left\{\frac{1}{n}\sum_{m=1}^n \frac{1}{2}\dist(X_m,y)^2: \dist(y,\origin)>\rho_1\right\}}
\]
converges to \(1\) as \(n\to\infty\). To establish this, we once again adapt methods from the proof of \citet[Theorem 2.3]{PatrangenaruBhattacharya-2003}.
First observe that we can apply the Cauchy-Schwartz inequality to show that
\begin{align*}
 &\frac{1}{n}\sum_{m=1}^n \frac{1}{2}\dist(X_m,y)^2\quad\geq\quad
\frac{1}{n}\sum_{m=1}^n \frac{1}{2}\left(\dist(X_m,\origin)-\dist(y,\origin)\right)^2
\quad\geq\quad\\
&\frac{1}{2}\dist(y,\origin)^2
+
\frac{1}{n}\sum_{m=1}^n \frac{1}{2}\dist(X_m,\origin)^2
-
\sqrt{2}\dist(y,\origin)\sqrt{\frac{1}{n}\sum_{m=1}^n \frac{1}{2}\dist(X_m,\origin)^2}\,.
\end{align*}
As before, if \(\phi_n(\origin)=0\) for all \(n\) then the Corollary follows immediately. Otherwise 
from Theorem \ref{thm:wlln-real-nonnegative-case} and the local Lindeberg condition we know that
\begin{equation*}
  \frac{1}{\phi_n(\origin)}\sum_{m=1}^n\frac{1}{2}\dist(X_m,\origin)^2\quad\to\quad 1
\qquad\text{ in probability}\,,
\end{equation*}
and hence the growth condition \eqref{eq:energy-growth} shows that as \(n\to\infty\) so (for example)
\[
\Prob{\frac{1}{n}\sum_{m=1}^n\frac{1}{2}\dist(X_m,\origin)^2 \leq 2C^2}\quad\to\quad1\,.
\]
This can be applied as follows; if we choose \(\rho_1\) to exceed \(2C+\sqrt{2+4C^2}\) then, with probability increasing to \(1\)  as \(n\to\infty\),
\[
 \frac{1}{2}\dist(y,\origin)^2 - \sqrt{2}\dist(y,\origin)\sqrt{\frac{1}{n}\sum_{m=1}^n \frac{1}{2}\dist(X_m,\origin)^2} -1 \quad\geq\quad 0\,.
\]
once \(\dist(y,\origin)>\rho_1\). Consequently as \(n\to\infty\) so
\[
 \Prob{\frac{1}{n}\sum_{m=1}^n \frac{1}{2}\dist(X_m,\origin)^2+1\leq\inf\left\{\frac{1}{n}\sum_{m=1}^n \frac{1}{2}\dist(X_m,y)^2: \dist(y,\origin)>\rho_1\right\}
}
\]
converges to \(1\) as required.
\end{proof}

\section{Euclidean interlude}\label{sec:euclid}
Before we turn to the central limit theorem on Riemannian manifolds, it is helpful to prove a modest variant on 
the usual central limit theorem for independent Euclidean (vector-valued) random variables, which may be of independent interest,
and which could be argued to capture more precisely the conventional statistical use of the idea of a central limit theorem.
The reader will see that the arguments in this section are almost entirely classical (see for example \citealp{Feller-1966})
and the main issue is simply to formulate the result. However 
we give complete proofs since we have not been able to trace general forms of these results in the literature,
and also because the classical proofs must be adapted to the vector-valued nature of the summands.

A natural condition for central limit approximation for normalized partial sums of
\(d\)-dimensional mean-zero finite-variance independent random vectors \(Y_1\), \ldots, \(Y_n\), \ldots
is that they should
satisfy a variant of Lindeberg's condition: for each \(\varepsilon>0\), as \(n\to\infty\) so
\begin{equation}\label{eq:natural-lindeberg}
 \frac{1}{\phi_n}\sum_{m=1}^n\Expect{\|Y_m\|^2\;;\;\|Y_m\|^2>\varepsilon\phi_n}\quad\to\quad0\,.
\end{equation}
Here we abbreviate \(\phi_n=\sum_{m=1}^n\Expect{\tfrac{1}{2}\|Y_m\|^2}\); this parallels the \(\phi_n(\origin)\) used in Sections \ref{sec:consistency} and \ref{sec:clt}
and leads us to consider the normalized sums \((X_1+\ldots+X_n)/\sqrt{2\phi_n}\).
(The factor \(\tfrac{1}{2}\) is awkward in the Euclidean context, but eases details of calculations later in the geometric context of Section \ref{sec:clt}.)
Note that \eqref{eq:natural-lindeberg} corresponds exactly to the local condition of Lindeberg type \eqref{eq:lindeberg-local} for \(X_1\), \(X_2\), \ldots.
However it should be clear that \eqref{eq:natural-lindeberg} cannot be sufficient to establish weak convergence to normality of \((Y_1+\ldots+Y_n)/\sqrt{2\phi_n}\);
consider two-dimensional examples in which the sequence \(Y_1\), \(Y_2\), \ldots alternates between longer and longer stretches of \(\Law{Y_k}=(N(0,1),0)\)
versus longer and longer stretches of \(\Law{Y_k}=(0,N(0,1))\). So we cannot hope for a central \emph{limit} theorem (thus the Cram\'er-Wold device is inapplicable); however it is the case that in fact 
\eqref{eq:natural-lindeberg} implies a central \emph{approximation} theorem.

In order to describe the result we first recall that the topology of weak convergence of probability measures can be metrized using a truncated Wasserstein distance
\begin{equation}\label{eq:wasserstein}
 \widetilde{W}_1(\mu,\nu)\quad=\quad
\inf\{\Expect{1\wedge \|U-V\|}\;:\; \Law{U}=\mu, \Law{V}=\nu\}
\end{equation}
(see for example \citealp[Chapter 7]{Villani-2003a}). Moreover by Kantorovich-Rubinstein representation
\cite[Remark 7.5(i)]{Villani-2003a}
we may write
\begin{multline}\label{eq:wassersteinKV}
 \widetilde{W}_1(\mu,\nu)=
\sup\{\int f\d(\mu-\nu): f\text{ is Lip\((1)\) for distance }1\wedge\|x-y\|\}\,.
\end{multline}
We now consider when the law of \((Y_1+\ldots+Y_n)/\sqrt{2\phi_n}\)
draws ever closer to the matching (but varying) multivariate normal distribution as \(n\to\infty\):
\begin{theorem}[Lindeberg central approximation theorem for vector-valued random variables]\label{thm:central-approximation-theorem}
 Suppose that \(Y_1\), \ldots, \(Y_n\), \ldots are independent zero-mean random \(d\)-dimensional vectors with finite variance-covariance matrices and that
the above variant of Lindeberg's condition \eqref{eq:natural-lindeberg} is satisfied.
Then
\begin{equation*}\label{eq:central-approximation-theorem}
 \widetilde{W}_1\left(\frac{Y_1+\ldots+Y_n}{\sqrt{2\phi_n}},Z_n\right)\quad\to\quad0\,,
\end{equation*}
where \(\phi_n=\sum_{m=1}^n\Expect{\tfrac{1}{2}\|Y_m\|^2}\) and
\(Z_n\) has the multivariate \(d\)-dimensional normal distribution of zero mean and variance-covariance matrix \(V_n\),
with
\begin{equation}\label{eq:var-covar}
 u^\top V_n u \quad=\quad \frac{1}{\phi_n}\sum_{m=1}^n \Expect{\frac{1}{2}\langle u, Y_m\rangle^2}
\qquad\text{for all vectors \(u\).}
\end{equation}
\end{theorem}
\begin{remark}
 Note that the variance-covariance matrix \(V_n\) has unit trace.
\end{remark}

\begin{proof}
 The proof is based heavily on the classic proof of the Feller-Lindeberg central limit theorem using characteristic functions.
First of all, observe that it is a consequence of the variant Lindeberg condition that
\[
 \sup_{m=1,\ldots,n} \frac{1}{\phi_n}\Expect{\|Y_m\|^2} \quad\to\quad 0\,.
\]
For otherwise we can find a subsequence \(\{n_r\}\) and \(m_r\) in \(1\), \ldots, \(n_r\) such that
for some positive \(c>0\) we have
\(\Expect{\|Y_{m_r}\|^2}\geq c \phi_{n_r}\) for all \(r\),
and if we choose \(\varepsilon<c\) then this implies that
\[
\sum_{m=1}^{n_r}\Expect{\|Y_{m}\|^2\;;\;\|Y_{m}\|^2>\varepsilon\phi_{n_r}}
\geq
\Expect{\|Y_{m_r}\|^2\;;\;\|Y_{m_r}\|^2>\varepsilon\phi_{n_r}}
\geq
(c-\varepsilon)\phi_{n_r}\,.
\]
Choosing \(\varepsilon<c\), this contradicts the variant Lindeberg condition \eqref{eq:natural-lindeberg}. Thus we can choose \(N=N(u)\) large enough
that \(\Expect{\tfrac{1}{2}\langle u,Y_m\rangle^2}<\phi_n\) for all \(m=1,\ldots,n\) and all \(n\geq N(u)\).

Using independence, set
\[
 \Psi_n(u)\quad=\quad
\Expect{\exp\left(i\frac{\langle u,Y_1+\ldots+Y_n\rangle}{\sqrt{2\phi_n}}\right)}
\quad=\quad
\prod_{m=1}^n\Expect{\exp\left(i\frac{\langle u,Y_m\rangle}{\sqrt{2\phi_n}}\right)}\,.
\]
By estimates based on Taylor expansion \cite[Section 27]{Billingsley-1986},
\begin{multline*}
 \left|
1+i\frac{\langle u,Y_m\rangle}{\sqrt{2\phi_n}}-\frac{1}{2}\frac{\langle u,Y_m\rangle^2}{2\phi_n}
-\exp\left(i\frac{\langle u,Y_m\rangle}{\sqrt{2\phi_n}}\right)
\right|
\leq
% \quad\leq\quad
\left(\frac{\langle u,Y_m\rangle^2}{2\phi_n}\right)\wedge\left(\frac{|\langle u,Y_m\rangle|^3}{(2\phi_n)^{3/2}}\right)\\
\quad\leq\quad{\max\{1,\|u\|^3\}}\left\{
\frac{\|Y_m\|^2}{2\phi_n}\Indicator{\|Y_{m}\|^2>\varepsilon\phi_n} + \frac{\|Y_m\|^3}{(2\phi_n)^{3/2}}\Indicator{\|Y_{m}\|^2\leq\varepsilon\phi_n}\right\}\,.
\end{multline*}
Hence for $n\geq N(u)$,
\begin{multline*}
  \left|\Psi_n(u)
-
 \prod_{m=1}^n\left(1-\frac{1}{2}\frac{\Expect{\langle u,Y_m\rangle^2}}{2\phi_n}\right)\right|
\quad\leq\quad\\
{\max\{1,\|u\|^3\}}\left\{\frac{1}{2\phi_n}\sum_{m=1}^n\Expect{\|Y_m\|^2\;;\;\|Y_m\|^2>\varepsilon\phi_{n}}
+
\sqrt{\frac{\varepsilon}{2}} \sum_{m=1}^n\frac{\Expect{\|Y_m\|^2}}{2\phi_n}\right\}\\
\quad=\quad
{\max\{1,\|u\|^3\}}\left\{\frac{1}{2\phi_n}\sum_{m=1}^n\Expect{\|Y_m\|^2\;;\;\|Y_m\|^2>\varepsilon\phi_{n}}
+
\sqrt{\frac{\varepsilon}{2}}\right\}
\end{multline*}
(recalling the definition of \(\phi_n\) for the last step, and noting that for \(n\geq N(u)\) we know
that every \(\tfrac{1}{2}\tfrac{\Expect{\langle u,Y_m\rangle^2}}{2\phi_n}\) is bounded above by \(\tfrac{1}{2}\)).

Now invoke the inequality
\[
e^{-p/(1-p)}\quad\leq\quad 1-p \quad\leq\quad e^{-p}\,,
\]
valid for \(0\leq p<1\). Since \(\frac{1}{2}\frac{\Expect{\langle u,Y_m\rangle^2}}{2\phi_n}<\tfrac{1}{2}\)
when \(n\geq N(u)\),
\begin{multline*}
 0\quad\leq\quad
\log\prod_{m=1}^n\exp\left(-\frac{1}{2}\frac{\Expect{\langle u,Y_m\rangle^2}}{2\phi_n}\right) 
- 
\log\prod_{m=1}^n\left(1-\frac{1}{2}\frac{\Expect{\langle u,Y_m\rangle^2}}{2\phi_n}\right)\\
\quad\leq\quad
\sum_{m=1}^n \frac{1}{1-\frac{1}{2}\frac{\Expect{\langle u,Y_m\rangle^2}}{2\phi_n}} \left(\frac{1}{2}\frac{\Expect{\langle u,Y_m\rangle^2}}{2\phi_n}\right)^2
\quad\leq\quad
\frac{1}{2}\max_{m=1,\ldots n} \frac{\Expect{\langle u,Y_m\rangle^2}}{2\phi_n} 
\\
\quad\leq\quad
\frac{\|u\|^2}{2} \max_{m=1,\ldots n} \frac{\Expect{\|Y_m\|^2}}{2\phi_n} \quad\to\quad0\,.
\end{multline*}
Accordingly we may use \eqref{eq:var-covar} to deduce that if \(n\geq N(u)\) then
\begin{align}\label{eq:char-fn-bound}
&\left|\Psi_n(u)-\exp\left(-\frac{1}{2} u^\top V_n u\right)\right|
\quad\leq\quad{\max\{1,\|u\|^3\}\times A_n}\,,
\\
\label{eq:char-fn-bound2}
&A_n=
\sqrt{\frac{\varepsilon}{2}}
+
\frac{1}{2} \max_{m=1,\ldots n} \frac{\Expect{\|Y_m\|^2}}{2\phi_n}
+
\frac{1}{2\phi_n}\sum_{m=1}^n\Expect{\|Y_m\|^2\;;\;\|Y_m\|^2>\varepsilon\phi_{n}}\,.
\end{align}
Since \(\varepsilon\) can be chosen to be arbitrarily small, and the variant Lindeberg condition \eqref{eq:natural-lindeberg} implies the other quantities converge to \(0\),
it follows that \(|\Psi_n(u)-\exp\left(-\frac{1}{2} u^\top V_n u\right)|\) converges to \(0\) for each fixed \(u\). 

We now convert this relationship between characteristic functions into an inequality for the truncated Wasserstein distance between the corresponding distributions.
To this end we use a Parseval equality \cite[XV.3]{Feller-1966}:
\[
 e^{-i\langle u, t \rangle}\Psi_n(u) \quad=\quad \Expect{\exp\left(i\langle u, \frac{Y_1+\ldots+Y_n}{\sqrt{2\phi_n}}-t \rangle\right)}\,.
\]
We can multiply by the symmetric \(d\)-dimensional normal density of variance \(\sigma^{-2}\), integrate with respect to \(u\), and rearrange to obtain
\begin{multline}\label{eq:smoothed-density}
\frac{1}{(2\pi)^{d}} \int_{\mathbb{R}^d} e^{-i\langle u, t \rangle}\Psi_n(u) e^{-\sigma^2 |u|^2/2}\d u
\quad=\quad\\
\frac{1}{(2\pi\sigma^{2})^{d/2}}\Expect{ \exp\left(-\frac{1}{2\sigma^2}\left\|t-\frac{Y_1+\ldots+Y_n}{\sqrt{2\phi_n}}\right\|^2\right)}\,. 
\end{multline}
The right-hand side (viewed as a function of \(t\)) is the density of \(\tfrac{Y_1+\ldots+Y_n}{\sqrt{2\phi_n}}+Z'\), 
where \(Z'\) has a \(d\)-dimensional multivariate normal distribution of variance-covariance matrix \(\sigma^{2}\mathbb{I}_d\),
independent of \(\tfrac{Y_1+\ldots+Y_n}{\sqrt{2\phi_n}}\). By the definition \eqref{eq:wasserstein} of Wasserstein distance the truncated Wasserstein distance between the distribution of \(\tfrac{Y_1+\ldots+Y_n}{\sqrt{2\phi_n}}\)
and the distribution of \(\tfrac{Y_1+\ldots+Y_n}{\sqrt{2\phi_n}}+Z'\) is bounded by
\[
 \Expect{\|Z'\|} \quad\leq\quad \text{constant} \times \sigma\,.
\]
Given any \(\eta>0\), we can choose \(\sigma\) to make this smaller than \(\eta/5\).

Choose \(Z_n\) to be of \(d\)-dimensional multivariate normal distribution with variance-covariance matrix \(V_n\), independent of \(Z'\).
The truncated Wasserstein distance between the distributions \(Z_n\) and \(Z_n+Z'\)
satisfies the same bound of \(\eta/5\). So consider bounds on the truncated Wasserstein distance between the distributions of 
(a) \(\tfrac{Y_1+\ldots+Y_n}{\sqrt{2\phi_n}}+Z'\), with density given by \eqref{eq:smoothed-density},
and (b) \(Z_n+Z'\)
whose density satisfies a similar formula but with the normal characteristic function
\(\exp\left(-\frac{1}{2} u^\top V_n u\right)\) replacing \(\Psi_n(u)\). By
the Kantorovich-Rubinstein representation \eqref{eq:wassersteinKV} of the truncated Wasserstein distance we may consider
\[			
\left|\Expect{f(\tfrac{Y_1+\ldots+Y_n}{\sqrt{2\phi_n}}+Z')} - \Expect{f(Z_n+Z')}\right|
\]
where \(f\) is Lip\((1)\) with respect to the \emph{truncated} distance function \(1\wedge\|x-y\|\) (see \eqref{eq:wassersteinKV}). 
Without loss of generality we take \(f(\origin)=0\);
the Lipschitz condition then implies that \(|f|\leq 1\) (since the truncated distance \(1\wedge\|x-y\|\) is always bounded above by \(1)\). 
Now both \(\tfrac{Y_1+\ldots+Y_n}{\sqrt{2\phi_n}}+Z'\) and \(Z_n+Z'\) have
variance-covariance matrices with traces bounded above by \(1+\sigma^2 d\); therefore once \(\sigma\) is fixed we may choose a large radius \(R\) and deduce by
Chebyshev that the distributions of both \(\tfrac{Y_1+\ldots+Y_n}{\sqrt{2\phi_n}}+Z'\) and \(Z_n+Z'\)
place probability mass of at most \(\eta/5\) outside the ball centred on \(\origin\) and of radius \(R\), so that
\begin{align*}
 |\Expect{f(\tfrac{Y_1+\ldots+Y_n}{\sqrt{2\phi_n}}+Z')\;;\; \|\tfrac{Y_1+\ldots+Y_n}{\sqrt{2\phi_n}}+Z'\|>R}| \quad&\leq\quad \eta/5\,,\\
 |\Expect{f(Z_n+Z')\;;\; \|Z_n+Z'\|>R}| \quad&\leq\quad \eta/5\,.
\end{align*}
Finally
\begin{multline*}
 \Big|\Expect{f(\tfrac{Y_1+\ldots+Y_n}{\sqrt{2\phi_n}}+Z')\;;\; \|\tfrac{Y_1+\ldots+Y_n}{\sqrt{2\phi_n}}+Z'\|\leq R} \\
- \Expect{f(Z_n+Z')\;;\; \|Z_n+Z'\|{\leq} R}\Big| \\
\quad\leq\quad
\int_{\ball(\origin,R)}\Big|\frac{1}{(2\pi)^{d}} \int_{\mathbb{R}^d} e^{-i\langle u, t \rangle}\Psi_n(u) e^{-\sigma^2 |u|^2/2}\d u\\
-\frac{1}{(2\pi)^{d}} \int_{\mathbb{R}^d} e^{-i\langle u, t \rangle}e^{-\frac{1}{2}u^\top V_n u} e^{-\sigma^2 |u|^2/2}\d u\Big|\d t\\
\quad\leq\quad
\frac{1}{(2\pi)^{d}} \int_{\ball(\origin,R)} \int_{\mathbb{R}^d}|\Psi_n(u)-e^{-\frac{1}{2}u^\top V_n u}|e^{-\sigma^2 |u|^2/2}\d u\d t\,.
\end{multline*}
Given \(\sigma\) and \(R\),
the dominated convergence theorem allows
us to choose \(N\) (\emph{not} depending on \(u\)) to make
this arbitrarily small for all \(n\geq N\), hence
\begin{multline*}
\Big|\Expect{f(\tfrac{Y_1+\ldots+Y_n}{\sqrt{2\phi_n}}+Z')\;;\; \|\tfrac{Y_1+\ldots+Y_n}{\sqrt{2\phi_n}}+Z'\|\leq R} \\
- \Expect{f(Z_n+Z')\;;\; \|Z_n+Z'\|{\leq} R}\Big|
\quad\leq\quad \eta/5\,,
\end{multline*}
for all \(n\geq N\). It therefore follows that for \(n\geq N\) we obtain
\[
 \widetilde{W}_1\left(\Law{\frac{Y_1+\ldots+Y_n}{\sqrt{2\phi_n}}},Z_n\right)\quad\leq\quad\eta\,,
\]
and since \(\eta>0\) was arbitrary the theorem follows.
\end{proof}
The following converse to this result mirrors Feller's converse to Lindeberg's theorem.
\begin{corollary}[Feller converse to Lindeberg central approximation theorem]\label{thm:Feller-converse}
In the situation of Theorem \ref{thm:central-approximation-theorem}, suppose that in place of the above variant of the Lindeberg condition
\eqref{eq:natural-lindeberg}
it is the case that
\begin{equation}\label{eq:feller-converse-condition}
 \frac{1}{\phi_n}\Expect{\| Y_n\|^2} \quad\to\quad 0\,,\qquad\text{ and }\qquad \phi_n\to\infty\,,
\end{equation}
and that
\begin{equation}\label{eq:wasserstein-two}
 \widetilde{W}_1\left(\frac{Y_1+\ldots+Y_n}{\sqrt{2\phi_n}},Z_n\right)\quad\to\quad0\,.
\end{equation}
Then the Lindeberg condition \eqref{eq:natural-lindeberg} must be satisfied.
\end{corollary}
\begin{proof}
As a consequence of \eqref{eq:feller-converse-condition} and the fact that \(\phi_n\) increases with \(n\),
\[
 \lim_{n\to\infty}\max_{1\leq m\leq n}\frac{\Expect{\| Y_m\|^2}}{\phi_n}\quad\leq\quad
 \lim_{n\to\infty}\max_{1\leq m\leq k}\frac{\Expect{\| Y_m\|^2}}{\phi_n}+ \lim_{n\to\infty}\max_{k< m}\frac{\Expect{\| Y_m\|^2}}{\phi_m}
\]
also tends to zero.
For fixed \(u\in\mathbb{R}^d\), 
the bounded Lipschitz nature of \(\exp(i\langle u, x\rangle)\) as a function of \(x\), 
applied to \eqref{eq:wasserstein-two} and the Kantorovich-Rubinstein characterization \eqref{eq:wassersteinKV}
together imply that
\[
 \Expect{\exp\left(i\frac{\langle u,Y_1+\ldots+Y_m\rangle}{\sqrt{2\phi_n}}\right)} - \exp\left(-\frac{1}{2}\langle u,V_n u\rangle\right)
\quad\to\quad0\,.
\]
Since \(V_n\) has unit trace, we can multiply through by \(\exp\left(\frac{1}{2}\langle u,V_n u\rangle\right)\), take logs and use independence to see that
\[
 \frac{1}{2}\langle u,V_n u\rangle +
 \sum_{m=1}^n\log\Expect{\exp\left(i\frac{\langle u,Y_m\rangle}{\sqrt{2\phi_n}}\right)} \quad\to\quad 0\,.
\]
Standard estimates using Taylor expansion show that
\begin{multline*}
 \left|\log\Expect{\exp\left(i\frac{\langle u,Y_m\rangle}{\sqrt{2\phi_n}}\right)}-\left(\Expect{\exp\left(i\frac{\langle u,Y_m\rangle}{\sqrt{2\phi_n}}\right)}-1\right)\right|\\
\quad\leq\quad
\left|\Expect{\exp\left(i\frac{\langle u,Y_m\rangle}{\sqrt{2\phi_n}}\right)}-1\right|^2\,,
\end{multline*}
while
\begin{multline*}
 \sum_{m=1}^n\left|\Expect{\exp\left(i\frac{\langle u,Y_m\rangle}{\sqrt{2\phi_n}}\right)}-1\right|^2
\quad\leq\quad\\
\left(\max_{1\leq m\leq n}\left|\Expect{\exp\left(i\frac{\langle u,Y_m\rangle}{\sqrt{2\phi_n}}\right)}-1\right|\right)
\times
\sum_{m=1}^n\left|\Expect{\exp\left(i\frac{\langle u,Y_m\rangle}{\sqrt{2\phi_n}}\right)}-1\right|\\
\quad\leq\quad
\max_{1\leq m\leq n}\,\frac{\|u\|^2}{2}\frac{\Expect{\|Y_m\|^2}}{2\phi_n}\times\sum_{m=1}^n\frac{\|u\|^2}{2}\frac{\Expect{\|Y_m\|^2}}{2\phi_n}
=
\max_{1\leq m\leq n}\,\frac{\|u\|^4}{8}\frac{\Expect{\|Y_m\|^2}}{\phi_n}\to 0\,.
\end{multline*}
Thus for fixed \(u\)
\[
 \frac{1}{2}\langle u,V_n u\rangle -
 \sum_{m=1}^n\Expect{1-\exp\left(i\frac{\langle u,Y_m\rangle}{\sqrt{2\phi_n}}\right)} \quad\to\quad 0\,.
\]
Taking real parts and splitting the expectation at  \(\|Y_m\|^2=\varepsilon\phi_n\),
\begin{multline*}
  \frac{1}{2}\langle u,V_n u\rangle- 
 \sum_{m=1}^n\Expect{1-\cos\left(\frac{\langle u,Y_m\rangle}{\sqrt{2\phi_n}}\right)\;;\;\|Y_m\|^2\leq\varepsilon\phi_n}
\quad=\quad\\
 \sum_{m=1}^n\Expect{1-\cos\left(\frac{\langle u,Y_m\rangle}{\sqrt{2\phi_n}}\right)\;;\;\|Y_m\|^2>\varepsilon\phi_n} + o(1)\,,
\end{multline*}
where we must bear in mind that the \(o(1)\) term depends on \(u\).
The right-hand side is bounded above by
\[
 \sum_{m=1}^n\Expect{2\times\frac{\|Y_m\|^2}{\varepsilon\phi_n}\;;\;\|Y_m\|^2>\varepsilon\phi_n} + o(1)
\quad\leq\quad
\frac{4}{\varepsilon}+o(1)\,;
\]
while the left-hand side is bounded below by
\begin{multline*}
   \frac{1}{2}\langle u,V_n u\rangle -
 \sum_{m=1}^n\Expect{\frac{1}{2}\left(\frac{\langle u,Y_m\rangle}{\sqrt{2\phi_n}}\right)^2\;;\;\|Y_m\|^2\leq\varepsilon\phi_n}\\
\quad=\quad
\sum_{m=1}^n\Expect{\frac{1}{2}\left(\frac{\langle u,Y_m\rangle}{\sqrt{2\phi_n}}\right)^2\;;\;\|Y_m\|^2>\varepsilon\phi_n}\,
\end{multline*}
by \eqref{eq:var-covar}; thus
\[
\frac{1}{\phi_n}\sum_{m=1}^n\Expect{\left\langle \frac{u}{\|u\|},Y_m\right\rangle^2\;;\;\|Y_m\|^2>\varepsilon\phi_n}
\quad\leq\quad
\frac{1}{\|u\|^2}\left(\frac{16}{\varepsilon}+o(1)\right)\,.
\]
The variant Lindeberg condition \eqref{eq:natural-lindeberg} now follows by summing over vectors \(u/\|u\|\) forming an orthonormal basis,
and choosing suitably large \(\|u\|\).
\end{proof}

\section{Central limit theory for empirical Fr\'echet means}\label{sec:clt}

In order to discuss the second-order theory of empirical Fr\'echet means, namely central limit theorems, we augment the metric space structure of \(\MetricSpace\) 
by moving to the context of a complete and connected Riemannian manifold $\Manifold$ of dimension $d$.
Let $\dist(x,y)$ be the Riemannian distance between points \(x\), \(y\in\Manifold\). 
For any $x\in\Manifold$, let $\CutLocus{x}$ denote the cut locus of $x$. Let \(\Exp_x:T_x\Manifold\to\Manifold\) be the Exponential map from the tangent space \(T_x\Manifold\) to \(\Manifold\); 
observe that \(\Exp_x^{-1}(y)\) can be defined uniquely for \(y\not\in\CutLocus{x}\) by \(\Exp_x^{-1}(y)=\gamma'(1)\), 
where \(\gamma:[0,1]\to\Manifold\) is the unique minimal geodesic running from \(x\) to \(y\).
Now let \(\Parallel_{x,y}:T_x\Manifold\to T_y\Manifold\) be the parallel transport map along the geodesic \(\gamma\), and note that \(\Parallel_{x,y}^{-1}=\Parallel_{y,x}\),
both being defined when \(x\not\in\CutLocus{y}\) equivalently \(y\not\in\CutLocus{x}\).
Finally, denote the covariant derivative by $\nabla$: if \(U\) is a smooth vectorfield and \(\gamma\) is a geodesic then 
the covariant derivative of \(U\) at \(\gamma(0)\) in the direction \(\gamma'(0)\) is given by
\[
 \nabla_{\gamma'(0)} U \quad=\quad \lim_{s\downarrow0} \frac{\Parallel_{\gamma(s),\gamma(0)} U(\gamma(s)) - U(\gamma(0))}{s}\,.
\]
Moreover \(\nabla_{\gamma'(0)}\) depends only on the tangent vector \(\gamma'(0)\), rather than the actual curve \(\gamma\).

Our discussion concerns a sequence of independent
(but \emph{not} identically distributed)
random variables \(X_1\), \(X_2\), \ldots, taking values in $\Manifold$, for which each \(\Expect{\dist(x,X_i)^2}\) is finite for some (and therefore for all) \(x\), 
and which share a common Fr\'echet mean $\origin\in\Manifold$.
Furthermore we suppose that
\begin{eqnarray}
\Prob{X_n\in\CutLocus{\origin}}=0\quad\hbox{ for }\quad n\geqslant1\,.
\label{eq:cut-locus-condition}
\end{eqnarray}
For each \(n\) we choose \(\mathcal{E}(X_1,\ldots,X_n)\) to be a measurably selected empirical local Fr\'echet mean of \(X_1\), \ldots, \(X_n\), and we suppose it possible to make these choices so 
that \(\mathcal{E}(X_1,\ldots,X_n)\) converges to $\origin$ in probability. (Theorem \ref{thm:consistency} delineates a large class of cases in which this can be done.)

For each \(i\geq1\) we can define a random vectorfield \(Y_i\) on \(\Manifold\setminus\CutLocus{X_i}\) by
\begin{equation}\label{eq:random-vectorfield-Y}
  Y_i(x)\quad=\quad \Exp_x^{-1}(X_i)\,.
\end{equation}
Here we use the definition of \(\Exp_x^{-1}\) on \(\Manifold\setminus\CutLocus{x}\); 
in the cases when \(\Exp_x^{-1}(X_i)\) is not defined we choose \(Y_i\) measurably but otherwise arbitrarily from the pre-image of \(X_i\) under \(\Exp_x\).
In fact it can be shown that \eqref{eq:random-vectorfield-Y} defines \(Y_i(x)\) uniquely for almost all \(x\) with probability \(1\); moreover the cut locus
condition \eqref{eq:cut-locus-condition} ensures that \(Y_i(\origin)\) in particular is almost surely well-defined.

Since \(\origin\) is a Fr\'echet mean of each \(X_i\), it follows that \(\Expect{Y_i(\origin)}=0\); moreover the finiteness of \(\Expect{\dist(\origin,X_i)^2}\)
implies the finiteness of \(\Expect{\|Y_i\|^2}\), which is the trace of the variance-covariance matrix of the random vector \(Y_i\). Moreover the calculus of manifolds shows that
\begin{equation}
  Y_i(x)\;=\; -\dist(x,X_i)\grad_x\dist(x,X_i)
\,=\,
\grad_x\left(-\tfrac{1}{2}\dist(x,X_i)^2\right)\,.
\end{equation}
Indeed, if \(x\in\Manifold\setminus\CutLocus{X_i}\) then covariant differentiation defines a symmetric \((d\times d)\) tensor \(H_i(x)=-(\nabla Y_i)(x)\),
acting on vectorfields \(U\), \(V\) by
\begin{equation}\label{eq:definition-of-H}
 \langle H_i U, V\rangle(x)
\;=\;
\langle -\nabla_U Y_i, V\rangle(x)
\;=\;
\Hess_x\left(\tfrac{1}{2}\dist(x,X_i)^2\right)(U,V)\,.
\end{equation}
(The sign of \(H_i\) is chosen so that if \(X_i=\origin\) then \(H_i(\origin)\) is the identity tensor.)

As noted above, the assumption that $\origin$ is a Fr\'echet mean of $X_i$ for all $i\geqslant1$ implies that
$Y_i(\origin)=\Exp^{-1}_\origin(X_i)$ determines a sequence of independent random variables with zero mean on $T_\origin({\Manifold})$. 
Then Theorem \ref{thm:central-approximation-theorem} and Corollary \ref{thm:Feller-converse} capture the conditions under which
the normalized sum \((Y_1(\origin)+\ldots+Y_n(\origin))/\sqrt{2\phi_n(\origin)}\) is asymptotically multivariate normal
(where \(\phi_n\) is the aggregate energy function as defined in Section \ref{sec:basic}). Moreover a first-order Taylor expansion argument suggests that (under further regularity conditions)
the Exponential map of a suitable transformation of this normalized sum should approximate the local empirical Fr\'echet mean \(\mathcal{E}(X_1,\ldots,X_n)\); 
this corresponds to an application of Newton's root-finding method. 
This is indeed the case, and forms the main result of this section. 
However before we turn to this we must first prove a preliminary geometric result, required in order to control the effects of the approximation.

We begin by constructing a certain orthonormal frame field \(e_1\), \ldots, \(e_d\) over \(\Manifold\setminus\CutLocus{\origin}\).
Pick \(e_1(\origin)\), \ldots, \(e_d(\origin)\) to be an orthonormal basis for \(T_\origin(\Manifold)\), and extend by parallel
transport along minimal geodesics from \(\origin\) over all of \(\Manifold\setminus\CutLocus{\origin}\): \(e_r(x)=\Parallel_{\origin,x}e_r(\origin)\), for \(x\in\Manifold\setminus\CutLocus{\origin}\).
By the properties of geodesic normal coordinates, the vectorfields \(\nabla_{e_r}e_s\) all vanish at \(\origin\).

\begin{lemma}\label{lem:geometry}
 For given \(\varepsilon>0\), choose \(\rho>0\) such that \(\ball(\origin,\rho)\subseteq\Manifold\setminus\CutLocus{\origin}\) and
\(\|\nabla_{e_r}e_s\|<\varepsilon/d\) within \(\ball(\origin,\rho)\), for \(r\), \(s=1, \ldots, d\). Set \(Z_{r,i}=\langle Y_i,e_r\rangle e_r\), for
some \(Y_i\). 
Then (viewing \(\nabla Z_{r,i}\) as a symmetric \((d\times d)\) tensor) for \(x\in\ball(\origin,\rho)\) we have
\begin{multline}\label{eq:geometry}
 \|\Parallel_{x,\origin} \nabla Z_{r,i}(x)-\nabla Z_{r,i}(\origin)\|\quad\leq\quad
\left(1+2\varepsilon\rho\right)\sup_{x'\in\ball(\origin,\rho)}\|\Parallel_{x',\origin}\nabla Y_i(x')-\nabla Y_i(\origin)\|\\
+2\varepsilon\left(\|Y_i(\origin)\|+\|\nabla Y_i(\origin)\|\rho\right)\,.
\end{multline}
\end{lemma}
\begin{proof}
We suppress the dependence on the suffix \(i\) for the sake of convenience of exposition, and write \(Z_r=Z_{r,i}\), \(Y=Y_i\).
First consider \(\nabla_V Z_r\) for a general smooth vectorfield \(V\). By the calculus of covariant differentiation
\[
 \nabla_V Z_r\;=\; \nabla_V (\langle Y,e_r\rangle e_r)
\;=\;
\langle\nabla_V Y,e_r\rangle e_r + \langle Y,\nabla_V e_r\rangle e_r + \langle Y, e_r\rangle \nabla_V e_r\,.
\]
Because \(\nabla_V e_r\) vanishes at \(\origin\),
\begin{multline*}
\Parallel_{x,\origin}\nabla_V Z_r(x)-\nabla_V Z_r(\origin)
\quad=\quad\\
\left(\langle\nabla_V Y,e_r\rangle(x)-\langle\nabla_V Y_r,e_r\rangle(\origin)\right)e_r(\origin)+\\
+
\langle Y,\nabla_{V} e_r\rangle(x) e_r(\origin)
+
\langle Y,e_r\rangle(x) {\Parallel_{x,\origin}}\nabla_V e_r(x)\,.
\end{multline*}
The coefficient of \(e_r(\origin)\) in the first term on the right-hand side can be rewritten as the evaluation of \(\langle \Parallel_{x,\origin}\nabla_V Y-\nabla_V Y, e_r\rangle e_r\) at \(\origin\); 
the other two terms can be expanded to achieve 
\begin{multline*}
\Parallel_{x,\origin}\nabla_V Z_r(x)-\nabla_V Z_r(\origin)
\quad=\quad
\langle \Parallel_{x,\origin}\nabla_V Y-\nabla_V Y, e_r\rangle(\origin) e_r(\origin)
\\
+
\langle Y,\Parallel_{x,\origin}\nabla_V e_r\rangle(\origin)e_r(\origin)
+\langle \Parallel_{x,\origin}Y - Y,{\Parallel_{x,\origin}}\nabla_V e_r\rangle(\origin) e_r(\origin)
\\
+
\langle Y,e_r\rangle(\origin)(\Parallel_{x,\origin}\nabla_V e_r)(\origin)
+
\langle \Parallel_{x,\origin}Y - Y,e_r\rangle(\origin)(\Parallel_{x,\origin}\nabla_V e_r)(\origin)\,.
\end{multline*}
To control the size of the matrix \(M=\Parallel_{x,\origin}\nabla Z_r-\nabla Z_r\) at \(\origin\) we shall use the Frobenius norm \(\|M\|=\sqrt{M^\top M}=\sqrt{\sum_{i,j}M_{i,j}^2}\).
Now \(V\) is an arbitrary vectorfield, hence (evaluating tensor and vectorfields at \(\origin\) throughout) we may deduce that
\begin{multline}\label{eq:nearly-there}
 \|\Parallel_{x,\origin}\nabla Z_r(x)-\nabla Z_r(\origin)\| \quad\leq\quad\\
\|\Parallel_{x,\origin}\nabla Y(x) - \nabla Y(\origin)\| 
+2 \left(\|\Parallel_{x,\origin}Y(x)-Y(\origin)\|+\|Y(\origin)\|\right)\|\nabla e_r(x)\|
\\
\quad\leq\quad
\|\Parallel_{x,\origin}\nabla Y(x) - \nabla Y(\origin)\| 
+2 \varepsilon \left(\|\Parallel_{x,\origin}Y(x)-Y(\origin)\|+\|Y(\origin)\|\right)
\end{multline}
so long as \(x\in\ball(\origin,\rho)\).

We now apply the Mean Value Theorem to observe that
\begin{multline}\label{eq:mvt-result}
 \|\Parallel_{x,\origin}Y(x)-Y(\origin)\|
\quad\leq\quad\dist(x,\origin) \sup_{x'\in\ball(\origin,\rho)}\|\nabla Y(x')\|
\\
\quad\leq\quad
\dist(x,\origin) \left(\|\nabla Y(\origin)\|+\sup_{x'\in\ball(\origin,\rho)}\|\Parallel_{x',\origin}\nabla Y(x')-\nabla Y(\origin)\|\right)
\end{multline}
and thus (restoring the dependence on 
the suffix \(i\)) we can apply \eqref{eq:mvt-result} to \eqref{eq:nearly-there} and combine with \(\dist(\origin,x)\leq\rho\) to deduce the required inequality.
\end{proof}

The above lemma allows us to control the errors arising from the approximation implicit in the Newton method described above. 
We can now state and prove the main theorem of this section.
\begin{theorem}\label{thm:lindeberg-clt}
Suppose that \(X_1\), \(X_2\), \ldots are independent non-identically distributed random variables taking values in \(\Manifold\), such that for all \(n\) and all \(x\in\Manifold\)
the aggregate energy function \(\phi_n(x)=\sum_{i=1}^n\tfrac{1}{2}\Expect{\dist(x,X_i)^2}\) is finite.
Suppose that \(\origin\) is a local Fr\'echet mean of each of the \(X_i\) and moreover suppose that \(\Prob{X_i\in \CutLocus{\origin}}=0\) for each \(i\).
Let \(Y_i=\Exp_\origin^{-1}(X_i)\).
Let \(x_n=\mathcal{E}(X_1,\ldots,X_n)\) be a measurable choice of local empirical Fr\'echet means such that \(x_n\to\origin\) in probability.
Suppose that the following conditions hold:
\begin{enumerate}
 \item \label{item:linear-growth} \(\phi_n(\origin)\) is of at least linear growth, so \(\liminf_{n\to\infty} \tfrac{\phi_n(\origin)}{n}=C_1>0\)
for a finite positive constant \(C_1>0\);
 \item \label{item:local-geometry-control} For each sufficiently small \(\rho>0\), as \(n\to\infty\) so
\[
\frac{1}{\phi_n(\origin)}\sum_{i=1}^n\sup_{x'\in\ball(\origin,\rho)}\Expect{\|\Parallel_{x',\origin}H_i(x')-H_i(\origin)\|}\quad\to\quad0\,,
\]
where \(H_i\) is as given in \eqref{eq:definition-of-H};
 \item \label{item:Hessian-control} There is a finite constant \(C_2\) such that
\[
 \limsup_{n\to\infty} \frac{1}{\phi_n(\origin)}\sum_{i=1}^n \Expect{\|H_i(\origin)\|^2} \quad\leq\quad C_2\,;
\]
\item \label{item:limiting-adjustment} Let \(\widetilde{H}_n\) be the coordinate-wise expectation
\[
 \widetilde{H}_n\quad=\quad \frac{\Expect{H_1(\origin)+\ldots+ H_n(\origin)}}{2\phi_n(\origin)}\,.
\]
Then the symmetric matrix \(\widetilde{H}_n\) is asymptotically non-singular; there is a positive constant \(C_3>0\) with
\(\limsup_{n\to\infty}\|\widetilde{H}_n^{-1}\|\leq C_3\);
\item \label{item:lindeberg} Finally we require a condition of Lindeberg type: for each \(\varepsilon>0\), as \(n\to\infty\) so
\[
 \frac{1}{\phi_n(\origin)}\sum_{i=1}^n\Expect{\dist(\origin,X_i)^2;\dist(\origin,X_i)^2>\varepsilon\phi_n(\origin)}\quad\to\quad0\,.
\]
\end{enumerate}
Let \(\widetilde{Z}_n\) have the multivariate normal distribution with zero mean and variance-covariance matrix \(\widetilde{H}_n^{-1}V_n\widetilde{H}_n^{-1}\),
where \(V_n\) is the variance-covariance matrix of \(\tfrac{Y_1+\ldots+Y_n}{\sqrt{2\phi_n(\origin)}}\).
Then as \(n\to\infty\) so
\[
 \widetilde{W}_1\left(\sqrt{2\phi_n(\origin)}\Exp_\origin^{-1}\left(x_n\right),\widetilde{Z}_n\right)\quad\to\quad0\,.
\]
\end{theorem}
\begin{proof}
Begin by representing \(\sum_{i=1}^n Y_i(x)=\sum_{i=1}^n\grad_x(-\tfrac{1}{2}\dist(x,X_i)^2)\) by a first-order Taylor series expansion about \(\origin\):
if \(\gamma_{x}\) is a minimal geodesic begun at \(\origin\) and ending at \(x\in\Manifold\setminus\CutLocus{\origin}\) at unit time then
\begin{multline*}
\Parallel_{x,\origin} \sum_{i=1}^n Y_i(x)\quad=\quad
\sum_{i=1}^n Y_i(\origin) +
\sum_{i=1}^n \nabla_{\gamma_{x}'(0)} Y_i(\origin)
+ \Delta_n(x){\gamma_{x}'(0)}\\
\quad=\quad
\sum_{i=1}^n Y_i(\origin) -
 \sum_{i=1}^n H_i(\origin)\gamma_{x}'(0)
+ \Delta_n(x)\gamma_{x}'(0)\,,
\end{multline*}
where the Mean Value Theorem can be applied to show that the matrix correction term \(\Delta_n(x)\) can be written as
\begin{equation*}
 \Delta_n(x) U %\quad=\quad -\sum_{r=1}^d \sum_{i=1}^n \left(
% \langle H_i U, e_r\rangle(\gamma_{x}(\theta_r)) e_r(\gamma_{x}(\origin)) 
% -
% \langle H_i U, e_r\rangle(\origin)
% e_r(\origin)\right)\\
% \quad=\quad \sum_{r=1}^d \sum_{i=1}^n \left(\Parallel_{\gamma_{x}(\theta_r),\origin}\left(
% \langle \nabla_U Y_i, e_r\rangle(\gamma_{x}(\theta_r))
%  e_r(\gamma_{x}(\theta_r))\right) 
% -
% \langle \nabla_U Y_i, e_r\rangle(\origin)
% e_r(\origin)\right)\\
\quad=\quad
\sum_{r=1}^d \sum_{i=1}^n
\left(
\Parallel_{\gamma_{x}(\theta_r),\origin} \nabla_U Z_{r,i}(\gamma_{x}(\theta_r))
-
\nabla_U Z_{r,i}(\origin)
\right)
\end{equation*}
for \(Z_{r,i}=\langle Y_i,e_r\rangle e_r\) as defined in Lemma \ref{lem:geometry}, and for suitable \(0\leq \theta_1\), \ldots, \(\theta_d\leq 1\).
Choosing \(\rho>0\) given \(\varepsilon\) as in Lemma \ref{lem:geometry}, if \(x\in\ball(\origin,\rho)\) 
then
\begin{multline}\label{eq:error-term}
  \|\Delta_n(x)\| \quad\leq\quad 
\sum_{r=1}^d \sum_{i=1}^n
\Big(
(1+2\varepsilon \rho)\sup_{x'\in\ball(\origin,\rho)}\|\Parallel_{x',\origin}\nabla Y_i(x')-\nabla Y_i(\origin)\|\\
+
2\varepsilon(\|Y_i(\origin)\|+\rho\|\nabla Y_i(\origin)\|)
\Big)\,.
\end{multline}

Now choose \(x=x_n=\mathcal{E}(X_1,\ldots,X_n)\). Since \(\gamma'_{x}(0)=\Exp_\origin^{-1}(x)\), it follows that \(\Parallel_{x_n,\origin} \sum_{i=1}^n Y_i(x_n)=0\).
If \(x_n=\mathcal{E}(X_1,\ldots,X_n)\in\ball(\origin,\rho)\) then
\[
 0 \quad=\quad 
\sum_{i=1}^n Y_i(\origin) -
\left(\sum_{i=1}^n H_i(\origin)-\Delta_n(x_n)\right) \Exp_{\origin}^{-1}(x_n)
\,.
\]
Consequently, so long as \(\sum_{i=1}^n H_i(\origin)-\Delta_n(x_n)\) is invertible, we may write 
\begin{equation}
 x_n=\mathcal{E}(X_1,\ldots,X_n)
\;=\;
\Exp_{\textbf{o}}\left(\left(\sum_{i=1}^n H_i(\origin)-\Delta_n(x_n)\right)^{-1} \sum_{i=1}^n Y_i(\origin)\right)\,.
\end{equation}
Use the aggregrate energy function \(\phi_n(x)=\sum_{i=1}^n \Expect{\tfrac{1}{2}\dist( X_i,x)^2}\)
(defined in Section \ref{sec:consistency}) to adjust the above equation into a form hinting at a central limit approximation for \(\mathcal{E}(X_1,\ldots,X_n)\):
\begin{multline}
 \sqrt{2\phi_n(\origin)} \times \Exp_\origin^{-1}\left(\mathcal{E}(X_1,\ldots,X_n)\right)
\quad=\quad\\
\left(\sum_{i=1}^n \frac{H_i(\origin)}{2\phi_n(\origin)}-\frac{\Delta_n(x_n)}{2\phi_n(\origin)}\right)^{-1} \frac{\sum_{i=1}^n Y_i(\origin)}{\sqrt{2\phi_n(\origin)}}\,.
\end{multline}

Using our estimates on the Frobenius norm \(\|\Delta_n(x_n)\|\),
\begin{multline*}
 \frac{1}{d}\left\|\frac{\Delta_n(x_n)}{2\phi_n(\origin)}\right\|\quad\leq\quad
\left(1+2\varepsilon\rho\right)\frac{\sum_{i=1}^n\sup_{x'\in\ball(\origin,\rho)}\|\Parallel_{x',\origin}\nabla Y_i(x')-\nabla Y_i(\origin)\|}{2\phi_n(\origin)}
\\
+\varepsilon\frac{\sum_{i=1}^n\|Y_i(\origin)\|}{\phi_n(\origin)}+\varepsilon\rho\frac{\sum_{i=1}^n\|\nabla Y_i(\origin)\|}{\phi_n(\origin)}\,.
\end{multline*}
We are given that \(x_n\to\origin\) in probability, so with probability tending to \(1\) we may apply condition \ref{item:local-geometry-control} of the theorem to the first of these summands,
together with the Markov inequality, and deduce that
\[
 \left(1+2\varepsilon\rho\right)\frac{\sum_{i=1}^n\sup_{x'\in\ball(\origin,\rho)}\|\Parallel_{x',\origin}\nabla Y_i(x')-\nabla Y_i(\origin)\|}{2\phi_n(\origin)}\to0
\qquad\text{in probability.}
\]

Application of the Cauchy-Schwartz inequality to the second summand, together with the definition of the aggregate energy function,
the fact that \(\|Y_i(\origin)\|=\dist(\origin,X_i)\), and condition \ref{item:linear-growth} of the theorem, shows
that
\[
 \frac{\Expect{\sum_{i=1}^n\|Y_i(\origin)\|}}{\phi_n(\origin)}
\;\leq\;
 \sqrt{\frac{\Expect{\sum_{i=1}^n\|Y_i(\origin)\|^2}}{\phi_n(\origin)}}\sqrt{\frac{n}{\phi_n(\origin)}}
= \sqrt{\frac{2n}{\phi_n(\origin)}}
\leq \frac{2}{\sqrt{C_1}}\,.
\]
A similar argument, but using condition \ref{item:Hessian-control} of the theorem as well as condition \ref{item:linear-growth}, allows us to deduce that
\[
\frac{\Expect{\sum_{i=1}^n\|\nabla Y_i(\origin)\|}}{\phi_n(\origin)}\;=\;
\frac{\Expect{\sum_{i=1}^n\|H_i(\origin)\|}}{\phi_n(\origin)}
\leq
\sqrt{\frac{\Expect{\sum_{i=1}^n\|H_i(\origin)\|^2}}{\phi_n(\origin)}}\sqrt{\frac{n}{\phi_n(\origin)}}
\leq
\sqrt{\frac{C_2}{C_1}}\,.
\]
Once again we may use the assumption that \(x_n\to\origin\) in probability; 
it follows from this and the Markov inequality that we may choose \(\varepsilon=\varepsilon_n\) to decrease to zero 
in such a manner that
\[
\varepsilon_n\frac{\sum_{i=1}^n\|Y_i(\origin)\|}{\phi_n(\origin)}
+
\varepsilon_n\rho\frac{\sum_{i=1}^n\|\nabla Y_i(\origin)\|}{\phi_n(\origin)}
\quad\to\quad0
\qquad\text{in probability.}
\]
Accordingly it follows that the matrix error term is negligible:
\begin{equation}\label{eq:negligibility}
 \frac{\Delta_n(x_n)}{2\phi_n(\origin)} \quad\to\quad 0 \qquad\text{in probability.}
\end{equation}

Now consider the behaviour of \(\sum_{i=1}^n \tfrac{H_i(\origin)}{2\phi_n(\origin)}\). 
We can control the sum of the variances of the components of this matrix:
by independence, and the fact that variance is always bounded above by second moment, we deduce that the sum of variances is bounded above by
\[
 \sum_{i=1}^n \frac{\Expect{\|H_i(\origin)\|^2}}{4\phi_n(\origin)^2}
\]
which converges to zero by conditions \ref{item:linear-growth} and \ref{item:Hessian-control} of the theorem. Accordingly
\[
 \sum_{i=1}^n \frac{H_i(\origin)}{2\phi_n(\origin)} - \widetilde{H}_n \quad\to\quad 0 \qquad \text{in probability.}
\]
Condition \ref{item:limiting-adjustment} of the theorem, together with the negligibility of \(\tfrac{\Delta_n(x_n)}{\phi_n(\origin)}\) established above in \eqref{eq:negligibility},
implies that the probability of the following being invertible converges to \(1\):
\[
 \left(\sum_{i=1}^n \frac{H_i(\origin)}{2\phi_n(\origin)}-\frac{\Delta_n(x_n)}{2\phi_n(\origin)}\right)\,.
\]
Moreover we may deduce that
\begin{equation}\label{eq:matrix-transform}
 \left(\sum_{i=1}^n \frac{H_i(\origin)}{2\phi_n(\origin)}-\frac{\Delta_n(x_n)}{2\phi_n(\origin)}\right)^{-1} - \widetilde{H}_n^{-1}
\quad\to\quad 0 \qquad\text{in probability.}
\end{equation}

Finally we consider the asymptotic distributional behaviour of
\[
 \frac{\sum_{i=1}^n Y_i(\origin)}{\sqrt{2\phi_n(\origin)}}\,.
\]
The Lindeberg condition \ref{item:lindeberg} of the theorem translates directly into a condition of Lindeberg type on the \(Y_i\): 
since \(\|Y_i(\origin)\|=\dist(\origin,X_i)\), and since \(\Expect{Y_i(\origin)}=0\) as a consequence of \(\origin\) being a local Fr\'echet mean of \(X_i\), as \(n\to\infty\) so
\[
 \frac{1}{\phi(\origin)}\sum_{i=1}^n\Expect{\|Y_i(\origin)\|^2;\|Y_i(\origin)\|^2>\varepsilon\phi_n(\origin)}\quad\to\quad0\,.
\]
Now Theorem \ref{thm:central-approximation-theorem} shows that
\begin{equation*}
  \widetilde{W}_1\left(\frac{Y_1(\origin)+\ldots+Y_n(\origin)}{\sqrt{2\phi_n}},Z_n\right)\quad\to\quad0\,,
\end{equation*}
where
 \(Z_n\) has the multivariate \(d\)-dimensional normal distribution of zero mean and variance-covariance matrix \(V_n\).

The proof of the theorem is now completed by using observation \eqref{eq:matrix-transform}, since
properties of the Wasserstein distance allow us to deduce
\[
 \widetilde{W}_1\left(\left(\sum_{i=1}^n \frac{H_i(\origin)}{2\phi_n(\origin)}-\frac{\Delta_n(x_n)}{2\phi_n(\origin)}\right)^{-1}\frac{Y_1(\origin)+\ldots+Y_n(\origin)}{\sqrt{2\phi_n}},\widetilde{H}_n^{-1}Z_n\right)
\quad\to\quad0
\]
from the convergence in probability specified in \eqref{eq:matrix-transform}, together with the upper bound supplied by condition \ref{item:limiting-adjustment} of the theorem.
\end{proof}

We finish by looking at a few special cases. First, if we assume that the \(X_n\) are actually identically distributed, then $\phi_n(\origin)= \frac{n}{2}\Expect{\dist(\origin,X_1)^2}$. Accordingly, 
if $0<\Expect{\dist(\origin,X_1)^2}<\infty$ then the conditions \ref{item:linear-growth} and \ref{item:lindeberg} of Theorem \ref{thm:lindeberg-clt} hold trivially. Moreover, 
\[
\tilde H_n=\tilde H=\Expect{H_1(\origin)}/\Expect{\dist(\origin,X_1)^2}
\]
and  
\[
V_n=V=\Expect{\Phi_{\origin,X_1}}/\Expect{\dist(\origin,X_1)^2}\,,
\]
where $\Phi_{x,y}$ is the self-adjoint linear operator on $T_x({\Manifold})$ defined by
\begin{eqnarray}
\Phi_{x,y}:v\mapsto\langle\Exp^{-1}_x(y),\,v\rangle\,\Exp^{-1}_x(y).
\label{eqn2bb}
\end{eqnarray}
Hence, the following is a direct consequence of Theorem \ref{thm:lindeberg-clt}.

\begin{corollary}
Suppose that \(X_1\), \(X_2\), \ldots is a sequence of independent and identically distributed random variables on $\Manifold$ with finite  \(\Expect{\dist(x,X_1)^2}\).
Suppose that \(\origin\) is the local Fr\'echet mean of \(X_1\) and that \(\Prob{X_1\in \CutLocus{\origin}}=0\).
Let \(x_n=\mathcal{E}(X_1,\ldots,X_n)\) be a measurable choice of local empirical Fr\'echet means such that \(x_n\to\origin\) in probability. Assume that 
\begin{description}
\item[ ({\it i})]
\[
\lim_{\rho\rightarrow 0}\Expect{\sup_{x\in\ball(\origin,\rho)}\left\|\Pi_{x',\origin} 
   H_1(x')-H_1(\origin)\right\|}=0\,;
\]
\item[({\it ii})] $\Expect{\left\|H_1(\origin)\right\|^2}<\infty$;

\item[({\it iii})] $\Expect{H_1(\origin)}^{-1}$ exists.

\end{description}
Then we have the following weak convergence as $n\rightarrow\infty$:
\[
\sqrt{n}\Exp^{-1}_\origin(x_n)\buildrel d\over\longrightarrow \text{MVN}(0,\tilde H^{-1}\Expect{\Phi_{\origin,X_1}}\,\tilde H^{-1})\,,
\]
where the limit is the multivariate normal distributaion with zero mean and variance-covariance matrix \(\tilde H^{-1}\Expect{\Phi_{\origin,X_1}}\,\tilde H^{-1}\).
\label{cor0a}
\end{corollary}

If there exists a local coordinate chart $\psi(x)=(x_1(x),\ldots,x_d(x))$ with a domain which contains 
the support of the distribution of $X_1$, then let us write $(\xi_1,\ldots,\xi_d)$ and $(\zeta^n_1,\ldots,\zeta^n_d)$ 
respectively for the coordinates of $\Exp^{-1}_\origin(X_1)$ and $\Exp^{-1}_\origin(x_n)$ with respect to the 
basis $(\partial_{x_1},\ldots,\partial_{x_d})$ in $T_\origin({\Manifold})$. 
The following result is the version of Corollary \ref{cor0a} in terms of these coordinates.

\begin{corollary}
Write ${\bm\zeta}^n=(\zeta^n_1,\ldots,\zeta^n_d)^\top$. In the case of Corollary $\ref{cor0a}$, 
\[
\sqrt{n}{\bm\zeta}^n\buildrel{d}\over\longrightarrow \text{MVN}\left(0,\left(\Expect{H^\psi}\right)^{-1}G\,V_\psi G\,\left(\Expect{H^\psi}
   \right)^{-1}\right),\qquad\hbox{ as }n\rightarrow\infty\,,
\]
where $G=(\langle\partial_{x_j},\partial_{x_k}\rangle)$, $V_\psi=(\Expect{\xi_j\xi_k})$ and $H^\psi$ is the matrix of the linear operator $H_1(\origin)$ under the coordinate chart $\psi$ with
\[
H^\psi_{\ell k}=-\frac{\partial\xi_\ell}{\partial x_k}-\sum_{j=1}^d\Gamma_{kj}^\ell\xi_j
\]
and with $\Gamma_{ij}^k$ being the Christoffel symbols for the chosen coordinate chart.
\label{cor1}
\end{corollary}

If the coordinates $\psi$ are normal coordinates centred at $\origin$ corresponding to an orthonormal basis $(e_1,\ldots,e_d)$ of $T_\origin({\Manifold})$, 
then $G=I$, and $(\xi_1,\ldots,\xi_d)$ and $(\zeta_1^n,\ldots,\zeta^n_d)$ become the normal coordinates, centred at $\origin$, of $X_1$ and $x_n$ respectively. 
Moreover, under a normal coordinate system, all the Christoffel symbols disappear at the centre $\origin$ and $\xi_k=-\frac{1}{2}\nabla_{e_k}\dist(x,X_1)^2\big|_{x=\origin}$, 
where $\nabla_{e_k}$ acts on the first variable of $\dist^2$ under normal coordinates, 
Corollary \ref{cor1} recovers the result of \citet{BhattacharyaPatrangenaru-2005} at the Fr\'echet mean $\origin$. 

Finally, if $\Manifold$ either has constant sectional curvature $\kappa$ 
or is a K\"ahler manifold of constant holomorphic sectional curvature $\kappa$ then the operator $H_i(x)$ defined by \eqref{eq:definition-of-H} can be expressed explicitly.
In the former case
\[H_i(x):v\mapsto\frac{1-f_\kappa(\dist(x,X_i))}{\dist(x,y)^2}\Phi_{x,X_i}(v)+f_\kappa
   (\dist(x,X_i))\,v\]
and, in the latter,
\begin{eqnarray*}
H_i(x):&&\hspace{-.5cm}v\mapsto f_\kappa(\dist(x,X_i))\,v+\frac{1-f_{\kappa/4}(\dist(x,X_i))}{\dist(x,X_i)^2}\,\Phi_{x,X_i}(v)\\
&&+\frac{f_\kappa(\dist(x,X_i))-f_{\kappa/4}(\dist(x,X_i))}{\dist(x,X_i)^2}\,\Phi^\jmath_{x,X_i}(v),
\end{eqnarray*}
where
\begin{eqnarray*}
f_\kappa(s)=\left\{\begin{array}{ll}\sqrt{\vert\kappa\vert}\,s\frac{C_\kappa
   (s)}{S_\kappa(s)}&\kappa\not=0\\1&\kappa=0,\end{array}\right.\qquad
S_\kappa(s)=\left\{\begin{array}{ll}\sin(\sqrt{\kappa}s)&\kappa>0\\s&
   \kappa=0\\\sinh(\sqrt{-\kappa}s)&\kappa<0,\end{array}\right.
\end{eqnarray*}
$C_\kappa(s)=S'_\kappa(s)/\sqrt{\vert\kappa\vert}$, and where $\jmath$ is the tensor field of isometries $\jmath_x$ of the tangent spaces $T_x({\Manifold})$ 
such that $\jmath_x^2=-\hbox{id}$, $\Phi_{x,y}$ is defined by \eqref{eqn2bb} and $\Phi^\jmath_{x,y}$ is also defined by \eqref{eqn2bb} 
but with $\Exp_x^{-1}$ there replaced by $\jmath_x\circ \Exp_x^{-1}$. 
Note that the consequent expression for the operator $\Expect{H_i(\origin)}$ was obtained in \cite{BhattacharyaBhattacharya-2008} 
when $\Manifold$ has constant curvature and an upper bound has also been given in the same paper for general $\Manifold$ in term of the bound of its curvature.

\bibliographystyle{imsart-nameyear}
\bibliography{abbrev,SG-CLT}

\begin{thebibliography}{27}

\bibitem[\protect\citeauthoryear{Afsari}{2011}]{Afsari-2010a}
\begin{barticle}[author]
\bauthor{\bsnm{Afsari},~\bfnm{Bijan}\binits{B.}}
(\byear{2011}).
\btitle{{Riemannian $L^{p}$ center of mass: Existence, uniqueness, and
  convexity}}.
\bjournal{Proceedings of the American Mathematical Society}
\bvolume{139}
\bpages{655--655}.
\bdoi{10.1090/S0002-9939-2010-10541-5}
\end{barticle}
\endbibitem

\bibitem[\protect\citeauthoryear{Barbour and Gnedin}{2009}]{BarbourGnedin-2009}
\begin{barticle}[author]
\bauthor{\bsnm{Barbour},~\bfnm{Andrew~D}\binits{A.~D.}} \AND
  \bauthor{\bsnm{Gnedin},~\bfnm{A.~V}\binits{A.~V.}}
(\byear{2009}).
\btitle{{Small counts in the infinite occupancy scheme}}.
\bjournal{Electron. J. Probab}
\bvolume{14}
\bpages{365--384}.
\end{barticle}
\endbibitem

\bibitem[\protect\citeauthoryear{Bhattacharya and
  Bhattacharya}{2008}]{BhattacharyaBhattacharya-2008}
\begin{barticle}[author]
\bauthor{\bsnm{Bhattacharya},~\bfnm{Abhishek}\binits{A.}} \AND
  \bauthor{\bsnm{Bhattacharya},~\bfnm{Rabindra}\binits{R.}}
(\byear{2008}).
\btitle{{Statistics on Riemannian manifolds: asymptotic distribution and
  curvature}}.
\bjournal{Proceedings of the American Mathematical Society}
\bvolume{136}
\bpages{2959--2967}.
\bdoi{10.1090/S0002-9939-08-09445-8}
\end{barticle}
\endbibitem

\bibitem[\protect\citeauthoryear{Bhattacharya and
  Patrangenaru}{2003}]{PatrangenaruBhattacharya-2003}
\begin{barticle}[author]
\bauthor{\bsnm{Bhattacharya},~\bfnm{Rabindra}\binits{R.}} \AND
  \bauthor{\bsnm{Patrangenaru},~\bfnm{Vic}\binits{V.}}
(\byear{2003}).
\btitle{{Large sample theory of intrinsic and extrinsic sample means on
  manifolds -- I}}.
\bjournal{The Annals of Statistics}
\bvolume{31}
\bpages{1--29}.
\bdoi{10.1214/aos/1046294456}
\end{barticle}
\endbibitem

\bibitem[\protect\citeauthoryear{Bhattacharya and
  Patrangenaru}{2005}]{BhattacharyaPatrangenaru-2005}
\begin{barticle}[author]
\bauthor{\bsnm{Bhattacharya},~\bfnm{Rabindra}\binits{R.}} \AND
  \bauthor{\bsnm{Patrangenaru},~\bfnm{Vic}\binits{V.}}
(\byear{2005}).
\btitle{{Large sample theory of intrinsic and extrinsic sample means on
  manifolds -- II}}.
\bjournal{The Annals of Statistics}
\bvolume{33}
\bpages{1225--1259}.
\bdoi{10.1214/009053605000000093}
\end{barticle}
\endbibitem

\bibitem[\protect\citeauthoryear{Bhattacharya and
  Rao}{1976}]{BhattacharyaRao-1976}
\begin{bbook}[author]
\bauthor{\bsnm{Bhattacharya},~\bfnm{Rabindra}\binits{R.}} \AND
  \bauthor{\bsnm{Rao},~\bfnm{Ramaswamy~Ranga}\binits{R.~R.}}
(\byear{1976}).
\btitle{{Normal approximation and asymptotic expansions}}.
\bpublisher{John Wiley \& Sons Ltd}, \baddress{New York-London-Sydney}.
\end{bbook}
\endbibitem

\bibitem[\protect\citeauthoryear{Billingsley}{1986}]{Billingsley-1986}
\begin{bbook}[author]
\bauthor{\bsnm{Billingsley},~\bfnm{Patrick}\binits{P.}}
(\byear{1986}).
\btitle{{Probability and measure}}.
\bpublisher{John Wiley \& Sons Ltd}, \baddress{New York}.
\end{bbook}
\endbibitem

\bibitem[\protect\citeauthoryear{Chatterjee}{2008}]{Chatterjee-2008}
\begin{barticle}[author]
\bauthor{\bsnm{Chatterjee},~\bfnm{Sourav}\binits{S.}}
(\byear{2008}).
\btitle{{A new method of normal approximation}}.
\bjournal{The Annals of Probability}
\bvolume{36}
\bpages{1584--1610}.
\bdoi{10.1214/07-AOP370}
\end{barticle}
\endbibitem

\bibitem[\protect\citeauthoryear{Chow and Teicher}{2003}]{ChowTeicher-2003}
\begin{bbook}[author]
\bauthor{\bsnm{Chow},~\bfnm{Yuan~Shih}\binits{Y.~S.}} \AND
  \bauthor{\bsnm{Teicher},~\bfnm{Henry}\binits{H.}}
(\byear{2003}).
\btitle{{Probability theory: independence, interchangeability, martingales}}.
\bpublisher{Springer-Verlag}, \baddress{New York - Heidelberg - Berlin}.
\end{bbook}
\endbibitem

\bibitem[\protect\citeauthoryear{Corcuera and
  Kendall}{1999}]{CorcueraKendall-1999}
\begin{barticle}[author]
\bauthor{\bsnm{Corcuera},~\bfnm{Jose-Manuel}\binits{J.-M.}} \AND
  \bauthor{\bsnm{Kendall},~\bfnm{Wilfrid~S}\binits{W.~S.}}
(\byear{1999}).
\btitle{{Riemannian barycentres and geodesic convexity}}.
\bjournal{Math. Proc. Camb. Phil. Soc.}
\bvolume{127}
\bpages{253--269}.
\bdoi{http://dx.doi.org/10.1017/S0305004199003643}
\end{barticle}
\endbibitem

\bibitem[\protect\citeauthoryear{Feller}{1966}]{Feller-1966}
\begin{bbook}[author]
\bauthor{\bsnm{Feller},~\bfnm{William}\binits{W.}}
(\byear{1966}).
\btitle{{An Introduction to Probability Theory and its Applications, Volume
  2}}.
\bpublisher{John Wiley \& Sons Ltd}, \baddress{New York}.
\end{bbook}
\endbibitem

\bibitem[\protect\citeauthoryear{Fr\'{e}chet}{1948}]{Frechet-1948}
\begin{barticle}[author]
\bauthor{\bsnm{Fr\'{e}chet},~\bfnm{Maurice}\binits{M.}}
(\byear{1948}).
\btitle{{Les \'{e}l\'{e}ments al\'{e}atoires de nature quelconque dans un
  espace distanci\'{e}}}.
\bjournal{Annales de l'institut Henri Poincar\'{e}}
\bvolume{10}
\bpages{215--310}.
\end{barticle}
\endbibitem

\bibitem[\protect\citeauthoryear{Karcher}{1977}]{Karcher-1977}
\begin{barticle}[author]
\bauthor{\bsnm{Karcher},~\bfnm{H.}\binits{H.}}
(\byear{1977}).
\btitle{{Riemannian center of mass and mollifier smoothing}}.
\bjournal{Communications on Pure and Applied Mathematics}
\bvolume{30}
\bpages{509--541}.
\end{barticle}
\endbibitem

\bibitem[\protect\citeauthoryear{Kendall}{1990}]{Kendall-1990d}
\begin{barticle}[author]
\bauthor{\bsnm{Kendall},~\bfnm{Wilfrid~S}\binits{W.~S.}}
(\byear{1990}).
\btitle{{Probability, convexity, and harmonic maps with small image I:
  Uniqueness and fine existence}}.
\bjournal{Proceedings of the London Mathematical Society (Third Series)}
\bvolume{61}
\bpages{371--406}.
\end{barticle}
\endbibitem

\bibitem[\protect\citeauthoryear{Kendall}{1991a}]{Kendall-1991c}
\begin{binproceedings}[author]
\bauthor{\bsnm{Kendall},~\bfnm{Wilfrid~S}\binits{W.~S.}}
(\byear{1991}a).
\btitle{{Convex geometry and nonconfluent $\Gamma$-martingales I: Tightness and
  strict convexity}}.
In \bbooktitle{Stochastic Analysis, Proceedings, LMS Durham Symposium, 11th -
  21st July 1990}
(\beditor{\bfnm{Martin~T}\binits{M.~T.}~\bsnm{Barlow}} \AND
  \beditor{\bfnm{Nicholas~H}\binits{N.~H.}~\bsnm{Bingham}}, eds.)
\bpages{163--178}.
\bpublisher{Cambridge University Press}, \baddress{Cambridge}.
\end{binproceedings}
\endbibitem

\bibitem[\protect\citeauthoryear{Kendall}{1991b}]{Kendall-1991a}
\begin{barticle}[author]
\bauthor{\bsnm{Kendall},~\bfnm{Wilfrid~S}\binits{W.~S.}}
(\byear{1991}b).
\btitle{{Convexity and the hemisphere}}.
\bjournal{The Journal of the London Mathematical Society (Second Series)}
\bvolume{43}
\bpages{567--576}.
\end{barticle}
\endbibitem

\bibitem[\protect\citeauthoryear{Kendall}{1992a}]{Kendall-1992a}
\begin{barticle}[author]
\bauthor{\bsnm{Kendall},~\bfnm{Wilfrid~S}\binits{W.~S.}}
(\byear{1992}a).
\btitle{{Convex geometry and nonconfluent $\Gamma$-martingales II:
  Well-posedness and $\Gamma$-martingale convergence}}.
\bjournal{Stochastics and Stochastic Reports}
\bvolume{38}
\bpages{135--147}.
\end{barticle}
\endbibitem

\bibitem[\protect\citeauthoryear{Kendall}{1992b}]{Kendall-1992c}
\begin{barticle}[author]
\bauthor{\bsnm{Kendall},~\bfnm{Wilfrid~S}\binits{W.~S.}}
(\byear{1992}b).
\btitle{{The Propeller: A counterexample to a conjectured criterion for the
  existence of certain convex functions}}.
\bjournal{The Journal of the London Mathematical Society (Second Series)}
\bvolume{46}
\bpages{364--374}.
\end{barticle}
\endbibitem

\bibitem[\protect\citeauthoryear{Kendall and Le}{2010}]{KendallLe-2010}
\begin{bincollection}[author]
\bauthor{\bsnm{Kendall},~\bfnm{Wilfrid~S}\binits{W.~S.}} \AND
  \bauthor{\bsnm{Le},~\bfnm{Huiling}\binits{H.}}
(\byear{2010}).
\btitle{{Statistical Shape Theory}}.
In \bbooktitle{New Perspectives in Stochastic Geometry}
(\beditor{\bfnm{Wilfrid~S}\binits{W.~S.}~\bsnm{Kendall}} \AND
  \beditor{\bfnm{Ilya~S}\binits{I.~S.}~\bsnm{Molchanov}}, eds.)
\bchapter{10}
\bpages{348--373}.
\bpublisher{The Clarendon Press (Oxford University Press)}, \baddress{Oxford}.
\end{bincollection}
\endbibitem

\bibitem[\protect\citeauthoryear{Le}{2001}]{Le-2001}
\begin{barticle}[author]
\bauthor{\bsnm{Le},~\bfnm{Huiling}\binits{H.}}
(\byear{2001}).
\btitle{{Locating Fr\'{e}chet Means with Application to Shape Spaces}}.
\bjournal{Advances in Applied Probability}
\bvolume{33}
\bpages{324--338}.
\end{barticle}
\endbibitem

\bibitem[\protect\citeauthoryear{Le}{2004}]{Le-2004}
\begin{barticle}[author]
\bauthor{\bsnm{Le},~\bfnm{Huiling}\binits{H.}}
(\byear{2004}).
\btitle{{Estimation of Riemannian Barycentres}}.
\bjournal{LMS J. Comput. Math}
\bvolume{7}
\bpages{193--200}.
\end{barticle}
\endbibitem

\bibitem[\protect\citeauthoryear{Picard}{1994}]{Picard-1994}
\begin{barticle}[author]
\bauthor{\bsnm{Picard},~\bfnm{Jean}\binits{J.}}
(\byear{1994}).
\btitle{{Barycentres et martingales sur une vari\'{e}t\'{e}}}.
\bjournal{Ann. Inst. H. Poincar\'{e} Probab. Statist.}
\bvolume{30}
\bpages{647--702}.
\end{barticle}
\endbibitem

\bibitem[\protect\citeauthoryear{R\"{o}llin}{2011}]{Rollin-2011}
\begin{bmisc}[author]
\bauthor{\bsnm{R\"{o}llin},~\bfnm{Adrian}\binits{A.}}
(\byear{2011}).
\btitle{{Stein's method in high dimensions with applications}}.
\bnote{arXiv:1101.4454}.
\end{bmisc}
\endbibitem

\bibitem[\protect\citeauthoryear{Villani}{2003}]{Villani-2003a}
\begin{bbook}[author]
\bauthor{\bsnm{Villani},~\bfnm{C\'{e}dric}\binits{C.}}
(\byear{2003}).
\btitle{{Topics in optimal transportation}}.
\bseries{Graduate Studies in Mathematics}
\bvolume{58}.
\bpublisher{American Mathematical Society}, \baddress{Providence, RI}.
\end{bbook}
\endbibitem

\bibitem[\protect\citeauthoryear{Ziezold}{1977}]{Ziezold-1977}
\begin{barticle}[author]
\bauthor{\bsnm{Ziezold},~\bfnm{Herbert}\binits{H.}}
(\byear{1977}).
\btitle{{On expected figures and a strong law of large numbers for random
  elements in quasi-metric spaces}}.
\bjournal{Conf. Inf. Theory, Statist. Decision Functions, Random Processes}
\bvolume{A}
\bpages{591--602}.
\end{barticle}
\endbibitem

\bibitem[\protect\citeauthoryear{Ziezold}{1989}]{Ziezold-1989}
\begin{binproceedings}[author]
\bauthor{\bsnm{Ziezold},~\bfnm{Herbert}\binits{H.}}
(\byear{1989}).
\btitle{{On expected figures in the plane}}.
In \bbooktitle{Geobild '89 (Georgenthal, 1989)}.
\bseries{Math. Res.}
\bvolume{51}
\bpages{105--110}.
\bpublisher{Akademie-Verlag}, \baddress{Berlin}.
\end{binproceedings}
\endbibitem

\bibitem[\protect\citeauthoryear{Ziezold}{1994}]{Ziezold-1994}
\begin{barticle}[author]
\bauthor{\bsnm{Ziezold},~\bfnm{Herbert}\binits{H.}}
(\byear{1994}).
\btitle{{Mean Figures and Mean Shapes Applied to Biological Figure and Shape
  Distributions in the Plane}}.
\bjournal{Biometrical Journal}
\bvolume{36}
\bpages{491--510}.
\bdoi{10.1002/bimj.4710360409}
\end{barticle}
\endbibitem

\end{thebibliography}

\end{document}